\documentclass[a4paper, 12pt]{article}
\usepackage{amsmath}
\usepackage{amssymb}
\usepackage{amsthm, enumerate}
\usepackage[top=20truemm, left=17truemm, right=17truemm]{geometry}

\usepackage{amscd}
\newtheorem{dfn}{Definition}[section]
\newtheorem{thm}[dfn]{Theorem}
\newtheorem{lem}[dfn]{Lemma}
\newtheorem{remark}[dfn]{Remark}
\newtheorem{cor}[dfn]{Corollary}

\usepackage{xcolor}
\usepackage{amscd}
\usepackage{graphicx}
\usepackage[all]{xy}
\newcommand{\K}{\mathbb{K}}
\newcommand{\C}{\mathbb{C}}
\newcommand{\M}{\mathcal{M}}
\newcommand{\Q}{\mathcal{Q}}
\newcommand{\id}{\operatorname{id}}
\newcommand{\ad}{\operatorname{Ad}}

\title{Ergodic automorphisms on Kirchberg algebras}
\author{Kengo Matsumoto, Taro Sogabe}
\date{}
\begin{document}

\maketitle
\begin{abstract}
    Combining the theory of extensions of C*-algebras and the Pimsner construction,
    we show that every countable infinite discrete group admits an ergodic action on arbitrary unital Kirchberg algebra.
In the proof, we give a Pimsner construction realizing many unital subalgebras of a given unital Kirchberg algebra as the fixed point algebras of single automorphisms.
Furthermore, for amenable infinite discrete groups, we show that every point-wise outer action on arbitrary unital Kirchberg algebra has an ergodic cocycle perturbation with the help of Gabe--Szab\'{o}'s theorem and  Baum--Connes' conjecture.
\end{abstract}

\section{Introduction}
For a group action on a C*-algebra,
it would not be easy to capture the fixed point algebra via KK-theory.
For example, in the case of single automorphism $\alpha$ of a C*-algebra $A$ (i.e., $\mathbb{Z}$-action $\alpha : \mathbb{Z}\curvearrowright A$),
one can compute the K-theory of the crossed product $K_*(A\rtimes_\alpha \mathbb{Z})$ via the Pimsner--Voiculescu exact sequence but not of the fixed point algebra in general.
One can not expect any strong relations between $K_*(A^\alpha)$ and $K_*(A)$,
and one of the extremal case might be the situation where $A$ is unital and the fixed point algebra $A^\alpha:=\{a\in A\;|\; \alpha_g(a)=a\;\text{for all}\; g\in G\}$ is $\mathbb{C}$ (i.e., the $G$-action $\alpha$ is ergodic).

The ergodic actions on operator algebras have been studied for von Neumann algebras, and there are many previous results.
Here, we would like to refer  \cite{MV, EOA, EOA3} for the researches on the existence of ergodic actions on von Neumann algebras.
For C*-algebras,
there are several interesting constructions of ergodic actions on the Cuntz algebras (see \cite{DE, CP}).
In the case of the Cuntz algebras,
these ergodic actions are obtained by a kind of quasi-free automorphisms coming from the Pimsner construction.

The Pimsner construction allows us to find a unital Kirchberg algebra $A$ and a quasi-free automorphisms $\alpha$ with the prescribed K-theoretic data (see \cite{P, Kumjian, Mey} and also \cite{Kita, Katsu}).
In this paper,
we combine the Pimsner construction with the theory of extensions of C*-algebras to control the fixed point algebras.
The key idea (see Sec. \ref{keycon}) comes from our previous research on the realization of reciprocal Cuntz--Krieger algebras \cite[Proposition 3.1.]{MatsSo}.

First, we show the following result providing a new picture for the unital subalgebras of Kirchberg algebras. 
\begin{thm}\label{Mats}
Let $G$ be a countable infinite discrete group.
Let $A$ be a unital Kirchberg algebra, and let $B$ be a unital,  separable, nuclear C*-algebra with a unital embedding $\iota_0 : B\hookrightarrow A$.
Then, there exist a unital embedding $\iota_1 : B\hookrightarrow A$ with $KK(\iota_0)=KK(\iota_1)$ and a $G$-action $\gamma$ on $A$ satisfying 
\begin{enumerate}
\item $A^\gamma=\iota_1(B)$ and there is a $G$-equivariant conditional expectation $A\to A^\gamma$,
\item $\gamma_g$ is outer for every $g\in G\backslash \{e\}$ (i.e., $\gamma$ is point wise outer) and $(A, \gamma)\sim_{KK^G}(A, {\rm id})$.
\end{enumerate}
In particular, every unital Kirchberg algebra $A$ has an ergodic action with an invariant state.
\end{thm}
\begin{remark}
Our construction also shows that the choice of $\iota_1$ only depends on $\iota_0$ and independent of $G$ (see Rem. \ref{ish}).
One can not have $\iota_0=\iota_1$ in general because there are many unital nuclear C*-subalgebras without conditional expectations (see \cite[Theorem 1.1., Corollary 3.4.]{Sat}).
\end{remark}
\begin{remark}
For $G=\mathbb{Z}$,
the theorem shows that many unital subalgebras are realized as the fixed point algebras of  single automorphisms.
\end{remark}
\begin{remark}
    There is a work of Y. Suzuki \cite{Sz} controlling some non-nuclear fixed point algebra of $\mathcal{O}_2$ given by the reduced group C*-algebra.
    His approach is completely different from ours and our approach can not realize non-nuclear fixed point algebra because our construction always admits a conditional expactation onto the fixed point algebra.
    There are several results by Y. Suzuki on the intermediate subalgebras, and it might be meaningful to try to find some applications of our construction to these researches.
\end{remark}

The above theorem shows the existence of ergodic action for arbitrary countable infinite discrete group, but the construction is rather "non-equivariant" and resulting action satisfies $(A, \gamma)\sim_{KK^G}(A, \id)$ (i.e., trivial in a sense of $KK^G$-theory).

Next we investigate the ergodic action with a prescribed KK-theoretic data.
By \cite[Theorem 1]{EOA},
the finite groups can not act ergodically on a unital Kirchberg algebra because of the absence of trace. 
For  infinite discrete groups,
we show the following theorem.
We denote by $F : KK^G(-, -)\to KK(-, -)$ the forgetful functor (see Sec. \ref{prel1}).
\begin{thm}\label{fee}
Let $G$ be a countable infinite discrete amenable group, and let $(A, \alpha)$ be a unital Kirchberg G-algebra.
Then, there exists an ergodic, point--wise outer action $\gamma$ 
with a $\gamma$-invariant state and a $KK^G$-equivalence $k\in KK^G((A, \alpha), (A, \gamma))^{-1}$ satisfying $[1_A]_0\hat{\otimes} F(k)=[1_A]_0$.
\end{thm}
As a consequence,
the Gabe--Szab\'{o} theorem \cite[Theorem 6.2. (iii), Proposition 3.14.]{GS} yields the following.
\begin{cor}
Let $G$ be a countable infinite discrete amenable group, and let $(A, \alpha)$ be an arbitrary unital Kirchberg algebra with a point--wise outer $G$-action.
Then, there is an $\alpha$-cocycle $u_\cdot$ making $(A, \alpha^{u})$ ergodic, where $\alpha^u$ is defined by $(\alpha^u)_g:=\ad u_g\circ\alpha_g,\;\; g\in G$.
\end{cor}
Applying the above corollary for a single automorphism (i.e., $G=\mathbb{Z}$), we have the following.
\begin{cor}
Every aperiodic automorphism $\alpha$ (i.e., $\alpha^n$ is not inner for every $n\in\mathbb{Z}\backslash \{0\}$) of a unital Kirchberg algebra $A$ is perturbed to an ergodic automorphism $\operatorname{Ad} u\circ \alpha$ by some unitary $u\in A$.
\end{cor}

\begin{remark}
In the case of von Neumann algebras, A. Marrakchi and S. Vaes show a surprising  result \cite[Theorem B]{MV} that the subset of $\alpha$-cocycles providing ergodic actions of a countable infinite amenable group is dense in the set of all $\alpha$-cocycles.
In the case of C*-algebras,
there often appear several topological obstructions related to $\alpha$-cocycles.
So it might be not so bad to ask whether the similar statement as in \cite{MV} holds or there is a topological obstruction to the statement in the C*-algebra case.
\end{remark}
\section*{Aknowledgemant}
We would like to thank Masaki Izumi for informing us the reference \cite{DE} and the result \cite[Theorem 1]{EOA}.
K. Matsumoto is supported by JSPS KAKENHI Grant Number 24K06775.
T. Sogabe is supported by JSPS KAKENHI Grant Number 24K16934.

\section{Preliminaries}
\subsection{Notation}
Let denote by $A, B, C, D, E, \cdots$  C*-algebras.
For a unital C*-algebra $A$,
we denote by $1_A$ (resp. $U(A)$) the unit (resp. the unitary group).
We write $1_A : \C\ni x\mapsto x1_A\in A$ by abuse of notation.
For a non-unital C*-algebra $A$ and its automorphism $\alpha$,
we write $A^\sim:=A+\mathbb{C}1$ and $\alpha^\sim(a+\lambda 1):=\alpha(a)+\lambda 1$ for its unitization.
Let $\mathbb{K}=\mathbb{K}(H)$ be the algebra of compact operators on the separable infinite dimensional  Hilbert space $H$.
In this paper,
we basically consider separable, nuclear C*-algebras.
Let $\mathcal{M}(B\otimes\mathbb{K})$ (resp. $\mathcal{Q}(B\otimes\mathbb{K}):=\M(B\otimes\K)/(B\otimes\K)$) be the multiplier algebra (resp. the Calkin algebra).
We write $q_B : \M(B\otimes\K)\to\Q(B\otimes\K)$ the quotient map.
In this paper,
the group $G$ is always assumed to be countable, infinite, discrete.
We write $\mathbb{K}_G:=\mathbb{K}\otimes \mathbb{K}(l^2(G))=\mathbb{K}(l^2(\mathbb{N})\otimes l^2(G))$.
Let $\rho_g : l^2(G)\ni \delta_h\mapsto \delta_{hg^{-1}}\in l^2(G)$ be the right regular representation.
By abuse of notation, we write $\rho : G\ni g\mapsto {\rm id}_{\mathbb{K}}\otimes\operatorname{Ad} \rho_g
\in \operatorname{Aut}(\mathbb{K}_G)$ (i.e., $\rho_g=\id_\K\otimes\ad\rho_g\in\operatorname{Aut}(\mathbb{K}_G)$).
Note that $(\K_G, \rho)$ and $(\K_G^{\otimes 2}, \rho^{\otimes 2})$ are conjugate by the Fell absorption.
Let us denote by $\lambda_g : l^2(G)\ni\delta_h\mapsto \delta_{gh}\in l^2(G)$ the left regular representation,
and we also denote by $\lambda_g :\mathcal{O}_\infty=\mathcal{O}(l^2(G))\ni S_h\mapsto S_{gh}\in\mathcal{O}_\infty$ the quasi-free action of $G$ on the Cuntz algebra $\mathcal{O}_\infty$ coming from the left regular representation where $\{S_g\}_{g\in G}$ are generating isometries of $\mathcal{O}_\infty$ with mutually orthogonal ranges.
To avoid the confusion with the terminology ``amenable action'',
we say nuclear G-algebra $(A, \alpha)$ if $A$ is a nuclear C*-algebra with a G-action $\alpha$.
For a $G$-invariant projection $p\in\M(A)^\alpha$,
we write the corner embedding $i_p : pAp\to A$.

The reader may refer the basics on the Hilbert bimodule, KK-theory for \cite{B} and the equivariant KK-groups (via Cuntz--Thomsen picture) for \cite{GS, Thom}.
For a right Hilbert $E$-module $X$,
we denote by $\langle\zeta, \xi\rangle_X\in E$ the $E$-valued inner product with the convention
\[\langle\zeta a, \xi b\rangle_X=a^*\langle\zeta, \xi\rangle_X b,\quad a, b\in E,\;\;\zeta, \xi\in X.\]
We use the same convention for the inner products of the Hilbert spaces.
We denote by $\mathcal{L}_E(X)$ the set of adjointable $E$-linear operators on $X$,
and call the pair $(\varphi : E\to \mathcal{L}_E(X), X)$ of $*$-homomorphism $\varphi$ and the right Hilbert module $X$ the Hilbert $E-E$-bimodule (see \cite{B}).
We write $a\zeta b:=\varphi(a)(\zeta b)$ for $a, b\in E, \;\zeta\in X$.
Note that $\mathcal{L}_B(H\otimes B)=\M(B\otimes\K(H))$.
For Hilbert $E-E$-bimodules $X_1, X_2$, we denote by \[T_\zeta : X_2\ni\xi\mapsto \zeta\otimes\xi\in X_1\otimes_EX_2,\quad \zeta\in X_1,\;\xi\in X_2,\] the creation operator.
Inner product of $X_1\otimes_E X_2$ is given by
\[\langle x_1\otimes x_2, y_1\otimes y_2\rangle=\langle x_2, \langle x_1, y_1\rangle_{X_1} y_2\rangle_{X_2},\]
and one has 
\[T_\zeta^* (x_1\otimes x_2)=\langle \zeta, x_1\rangle_{X_1} x_2.\]
We write
\[\mathcal{F}(X):=E\Omega_X\oplus\bigoplus_{k=1}^\infty X^{\otimes_E k}\]
where $\Omega_X$ is the vacuum vector.
We denote by $T_\zeta, \; \zeta\in X$ the creation operator
\[T_\zeta : x\Omega_X \mapsto \zeta x\in X,\; x\in E,\]
\[T_\zeta : X^{\otimes_E k}\ni\eta\mapsto \zeta\otimes\eta \in X^{\otimes k+1}.\]
For $\zeta:=\zeta_1\otimes \zeta_2\cdots \otimes \zeta_k\in X^{\otimes_E k}$,
we write $T_\zeta:=T_{\zeta_1}T_{\zeta_2}\cdots T_{\zeta_k}$ and $|\zeta|:=k$ for short.


\subsection{KK-groups}\label{prel1}
The equivariant KK-theory is introduced by G. G. Kasparov in a very general setting (see \cite{Kas}),
and we only use the simplest case of his theory where the algebras are all ungraded and the Bott elements are non-equivariant.
We refer to \cite{B} for the basics of the KK-theory and the reader may also refer to \cite{GS, Thom} for the Cuntz--Thomsen picture of $KK^G$-group.

For a $*$-homomorphism $\varphi : A\to B$,
we denote by
\[KK(\varphi):=(\varphi, B, 0)\in KK(A, B)\]
the corresponding Kasparov module,
and we write 
\[I_A:=KK(\id_A),\;\; KK(\varphi)\otimes I_D:=KK(\varphi\otimes\id_D)\]
(cf \cite[Definition 2.5.]{Kas}).
For $G$-algebras $(A, \alpha), \; (B, \beta)$,
a pair of $*$-homomorphism $\varphi : A\to B$ and $\beta$-cocycle $u_\cdot:=\{u_g\}_{g\in G}$ (i.e., unitaries $u_g\in\M(B)$ satisfying $u_g\beta_g(u_h)=u_{gh}$) are called a cocycle morphism if $\varphi\circ\alpha_g=\ad u_g\circ \beta_g\circ\varphi$ holds (\cite[Def. 1.2.]{GS}).
The left $B$-module $B^u:=B$ equipped with a $G$-action $g\cdot b:=u_g\beta_g(b)$ gives a Kasparov module denoted by
\[KK^G(\varphi, u_\cdot):=(\varphi, B^u, 0)\in KK^G((A, \alpha), (B, \beta))\]
(see \cite[Prop. 1.12.]{GS}).
Let $KK^G((A, \alpha), (B, \beta))^{-1}$ be the subset of $KK^G$-equivalences.
We denote by $\mu_{\C}\in KK^G((\C, \id), (\K_G, \rho))$ the Morita bimodule
\[\C e\curvearrowright e\K_G \curvearrowleft \K_G\]
with the $G$-action $g\cdot ex:=ex(1_{\M(\K)}\otimes\rho_g^*)$, $x\in \K_G$, where $e$ is a minimal projection.
One has $\mu_{\C}\in KK^G((\C, \id), (\K_G, \rho))^{-1}$ and the inverse is given by
\[\K_G\curvearrowright \K_G e\curvearrowleft \C e\]
with the $G$-action $g\cdot xe:=(1_{\M(\K)}\otimes\rho_g)xe$.
\begin{lem}\label{mori}
Let $(A, \alpha)$ be a separable, G-algebra with a G-invariant projection $p\in\mathcal{M}(A)^{{\alpha}}$ satisfying $\overline{ApA}=A$.
Then,
the G-equivariant embedding $i_p : (pAp, \alpha)\to (A, \alpha)$ is a $KK^G$-equivalence.
\end{lem}
\begin{proof}
Let $r : pAp\to \mathcal{L}_A(pA)$ and $R : A\to \mathcal{L}_{pAp} (Ap)$ be the left multiplications, then the Kasparov modules $(r, pA, 0)$ and $(R, Ap, 0)$ are mutually inverses to each other (i.e., they are $KK^G$-equivalences) by the assumption $\overline{ApA}=A$.
Now it is easy to see
\[KK^G(i_p)= (r, pA, 0)\oplus (0, (1-p)A, 0)=(r, pA, 0)\in KK^G((pAp, \alpha), (A, \alpha))^{-1}.\]
\end{proof}

We denote by $e_{g,h}$ the matrix units of $\K(l^2(G))$ in the following lemma.
\begin{lem}[{see \cite[Sec. 2]{IM}}]\label{IM}
    For a cocycle morphism $(\varphi, u_\cdot) : (A, \alpha)\to (B, \beta)$,
    there exists a unitary
    \[W:=\sum_{g\in G} u_{g^{-1}}\otimes 1_{\M(\K)}\otimes e_{g,g}\in \M(B\otimes\K_G)\]
    making the following diagram commute
    \[\xymatrix{
    (A, \alpha)\ar[r]^{KK^G(\varphi, u)}\ar[d]^{I_A\otimes\mu_{\C}}&(B, \beta)\ar[d]^{I_B\otimes\mu_{\C}}\\
    (A\otimes \K_G, \alpha\otimes\rho)\ar@<-0.8 ex>[r]_{KK^G(\ad W^*\circ(\varphi\otimes\id), 1)}&(B\otimes \K_G, \beta\otimes\rho).
    }\]
\end{lem}
\begin{proof}
    Since 
    \[KK^G(\varphi, u)\hat{\otimes} (I_B\otimes\mu_{\C})=(I_A\otimes\mu_{\C})\hat{\otimes} (KK^G(\varphi, u)\otimes I_{\K_G})\]
    (see \cite[Theorem 2.14, 8)]{Kas}),
    it is enough to show
    \[KK^G(\varphi\otimes\id, u\otimes 1)=KK^G(\ad W^*\circ (\varphi\otimes \id), 1)\in KK^G((A\otimes\K_G, \alpha\otimes\rho), (B\otimes\K_G, \beta\otimes\rho)).\]
    By definition,
    one has 
    \begin{align*}
    u_g\otimes 1=&u_g\otimes 1_{\M(\K)}\otimes 1_{\M(\K)}\\
    =&\sum_{h\in G}u_{gh^{-1}}\beta_g(u^*_{h^{-1}})\otimes 1_{\M(\K)}\otimes e_{hg^{-1}, hg^{-1}}\\
    =&W(\beta_g\otimes\rho_g)(W^*),
    \end{align*}
    the unitary $W$ gives a $G$-equivariant isomorphism of left $B$-module
    \[(B\otimes\K_G)^{u\otimes 1}\ni x\mapsto W^*x\in (B\otimes\K_G)^1\]
    which implies
    \begin{align*}KK^G(\varphi\otimes\id, u\otimes 1)=&(\varphi\otimes\id, (B\otimes\K_G)^{u\otimes 1}, 0)\\
    \cong& (\ad W^*\circ (\varphi\otimes\id), (B\otimes\K_G)^1, 0)\\
    =&KK^G(\ad W^*\circ (\varphi\otimes\id), 1).\end{align*}
\end{proof}

We denote by
\[F : KK^G((A, \alpha), (B, \beta))\ni KK^G(\varphi, u_\cdot)\mapsto KK(\varphi)\in KK(A, B)\]
the forgetful functor.
For the trivial actions,
the definition of Kasparov module gives the map
\[KK(A, B)\ni KK(\varphi)\mapsto KK^G(\varphi, 1)\in KK^G((A, \id), (B, \id))\]
which is compatible with the Kasparov product.
Thus, one has the non-equivariant Bott element
\[b\in KK(\C, S^2)^{-1}\subset KK^G((\C, \id), (S^2, \id))^{-1},\]
where $S^2:=C_0(0, 1)\otimes C_0(0, 1)$.

\subsection{Kirchberg algebras}
In this paper,
the Kirchberg algebras are separable, nuclear, simple, purely infinite C*-algebras.
We do not assume the UCT for the Kirchberg algebras.
The reader may refer to \cite{Ro} for more details.
Thanks to the Kirchberg--Phillips' theorem,
the isomorphism classes of the Kirchberg algebras are completely determined by KK-theory.
We refer \cite{Kirchberg, Phillips, Ro, gabe} for the classification of Kirchberg algebras.
\begin{thm}[{Kirchberg--Phillips' theorem}]\label{KP}
    Let $A, B$ be unital Kirchberg algebras.
    There exists an isomorphism $\theta : A\to B$ if and only if there exists a KK-equivalence $k\in KK(A, B)^{-1}$ satisfying $KK(1_A)\hat{\otimes}k=KK(1_B)$. 
The isomorphism is chosen to satisfy $KK(\theta)=k$. 
\end{thm}
By the recent breakthrough of J. Gabe and G. Szab\'{o},
we also have the equivariant version of the above theorem (see \cite{GS} for more general statement).
\begin{thm}[{\cite[Theorem 6.2.]{GS}}]
    Let $G$ be a countable discrete amenable group, and let $(A, \alpha), (B, \beta)$ be unital Kirchberg algebras with point-wise outer actions (i.e., $\alpha_g, \beta_g$ are not inner for $g\in G\backslash \{e\}$).
    If there exists a $KK^G$-equivalence $k\in KK^G((A, \alpha), (B, \beta))^{-1}$ satisfying $KK(1_A)\hat{\otimes} F(k)=KK(1_B)$,
    then there is an isomorphism $\theta : A\to B$ and a $\beta$-cocycle $u_\cdot$ satisfying $\theta\circ\alpha_g=\ad u_g\circ\beta_g\circ\theta$, $k=KK^G(\theta, u_\cdot)$.
\end{thm}
We will use the following existence theorem to construct our "Busby invariant" in Sec. \ref{busby}.
\begin{thm}[{\cite[Theorem 5.5.]{GS}}]\label{yeah}
    Let $G$ be a countable discrete amenable group.
    Let $A$ be separable exact C*-algebra with an action $\alpha : G\curvearrowright A$, and let $B$ be a Kirchberg algebra with point-wise outer action $\beta : G\curvearrowright B$.
    For every $x\in KK^G((A, \alpha), (B\otimes\K, \beta\otimes\id))$,
    there exists an injective cocycle morphism $(\varphi, u_\cdot) : (A, \alpha)\hookrightarrow (B\otimes\K, \beta\otimes\id)$ satisfying $x=KK^G(\varphi, u_\cdot)$.
\end{thm}

\subsection{Ext-groups and exact triangles}\label{nati}
We refer to \cite{B} for the basics of the extension of C*-algebras.
For separable, nuclear C*-algebras $A, B$,
an extension of $A$ by $B\otimes\K$ is a short exact sequence of upper sequence in the diagram
\[\xymatrix{B\otimes\K\ar@{=}[d]\ar[r]& E\ar[r]^{\pi_E}\ar[d]&A\ar[d]^{\tau}\\
B\otimes\K\ar[r]&\M(B\otimes\K)\ar[r]^{q_B}&\Q(B\otimes\K)}\]
where $q_B$ is the quotient map.

The map $\tau$ is called the Busby invariant and the extension is called unital (resp. essential) if $\tau$ is unital (resp. injective).
For an essential extension,
we identify $E$ with a subalgebra of $\M(B\otimes\K)$ and write $\pi_E=\tau^{-1}\circ q_B$.
For the nuclear C*-algebra $A$,
the set of equivalence classes of Busby invariants is denoted by $\operatorname{Ext}^1(A, B\otimes\K)$ (see \cite[Sec. 15]{B}).

We refer to \cite[Sec. 2, Appendix]{MN} for  the category $KK^G$ whose morphism set is given by the group $KK^G((A, \alpha), (B, \beta))$ for objects $(A, \alpha), (B, \beta)$.
A sequence
\[(SA, S\alpha)\to (C, \gamma)\to (B, \beta)\to (A, \alpha)\]
in the category $KK^G$ is called an exact triangle
if there exist $G$-equivariant $*$-homomorphism $f : B'\to A'$ and $KK^G$ equivalences $\eta_A\in KK^G((A, \alpha), (A', \alpha'))^{-1},\; \eta_B\in KK^G((B, \beta), (B', \beta'))^{-1},\; \eta_C\in KK^G((C, \gamma), (\operatorname{Cone}(f), \sigma))^{-1}$ making the following diagram commute
\[\xymatrix{
(SA, S\alpha)\ar[r]\ar[d]^{S\eta_A}&(C, \gamma)\ar[r]\ar[d]^{\eta_C}&(B, \beta)\ar[r]\ar[d]^{\eta_B}&(A, \alpha)\ar[d]^{\eta_A}\\
(SA', S\alpha')\ar[r]&(\operatorname{Cone}(f), \sigma)\ar[r]&(B', \beta')\ar[r]^{KK^G(f)}&(A', \alpha'),
}\]
where we write
\[\operatorname{Cone}(f):=\{(a(t), b)\in (C_0(0, 1]\otimes A')\oplus B'\;|\; a(1)=f(b)\},\]
\[\sigma_g (a(t), b):=(\alpha'_g(a(t)), \beta'_g(b)), \;\;\; g\in G.\]
We write $ev_1 : \operatorname{Cone}(f)\ni (a(t), b)\mapsto b\in B$.
For a nuclear C*-algebra $A$ and its extension $B\otimes\K\xrightarrow{l_E} E\xrightarrow{\pi_E} A$,
we write \[j(E) : B\otimes\K\ni x\mapsto (0, x)\in \operatorname{Cone}(\pi_E),\]
\[\iota(E) : SA\ni a(t)\mapsto (a(t), 0)\in \operatorname{Cone}(\pi_E),\]
and it is well-known that $j(E)$ is a $KK$-equivalence providing the following exact triangles (see \cite{CS}, \cite[Sec. 19.5.]{B}, \cite[Sec. 2.3.]{MN})
\[\xymatrix{
SA\ar@{-->}[r]\ar@{=}[d]&B\otimes\K\ar[r]^{KK(l_E)}\ar[d]^{KK(j(E))}&E\ar[r]^{KK(\pi_E)}\ar@{=}[d]&A\ar@{=}[d]\\
SA\ar@<-0.8 ex>[r]_{KK(\iota(E))}&\operatorname{Cone}(\pi_E)\ar@<-0.8 ex>[r]_{KK(ev_1)}&E\ar@<-0.8 ex>[r]_{KK(\pi_E)}&A.
}\]
For a compact group $H$ and an extension of nuclear $H$-C*-algebras $(J, \theta)\xrightarrow{l_F} (F, \theta)\xrightarrow{\pi_F} (F/J, \bar{\theta})$ with the $H$-equivariant maps,
nuclearlity gives a completely positive, contractive section $\rho : F/J\to F$ satisfying $\pi_F\circ \rho=\id_{F/J}$ and the compactness of $H$ allows us to modify $\rho$ to the equivariant section
$\tilde{\rho}(\cdot):=\int_H \theta_{h^{-1}}(\rho(\bar{\theta}_h(\cdot)))dh$.
Combining the Baum--Connes conjecture \cite{HigKas}, \cite[Theorem 8.5.]{MN} with \cite[p186, Theorem 5.4.]{Schochet},
we obtain the following.
\begin{thm}[{cf \cite[Sec. 2.3.]{MN}}]\label{wkeq}
Let $G$ be countable discrete amenable group, and let $(J, \theta)\xrightarrow{l_F} (F, \theta)\xrightarrow{\pi_F} (F/J, \bar{\theta})$ be an extension of nuclear $G$-algebras with $G$-equivariant $*$-homomorphisms.
Then, $KK^G(j(F))$ is a $KK^G$-equivalence.
In particular,
$J\to F\to F/J$ extends to an exact triangle
\[S(F/J)\xrightarrow{KK^G(\iota(F))\hat{\otimes}KK^G(j(F))^{-1}} J\xrightarrow{l_F} F\xrightarrow{\pi_F} F/J\] in $KK^G$.
\end{thm}
\begin{proof}
    By the above observation and \cite[Theorem 5.4.]{Schochet},
    the restriction to any compact (finite) subgroup $H\subset G$ gives $KK^H(j(F))\in KK^H(\operatorname{Res}_G^H J, \operatorname{Res}_G^H\operatorname{Cone}(\pi_F))^{-1}$ (i.e., $j(F)$ is a weak equivalence).
    Thus, \cite[Theorem 8.5.]{MN} shows $KK^G(j(F))\in KK^G(J, \operatorname{Cone}(\pi_F))^{-1}$.
\end{proof}
For an exact triangle $SA\to C\to B\xrightarrow{\eta} A$,
the triangle $SB\xrightarrow{-S\eta}SA\to C\to B$ is also an exact triangle.
We will use the following weak functoriality of the exact triangles.
\begin{lem}[{cf \cite[Sec. 2]{MN}}]\label{benri}
    For the commutative diagram of the following exact triangles
    \[\xymatrix{
    (SA, S\alpha)\ar[r]\ar[d]^{S\xi_A}&(C, \gamma)\ar[r]\ar@{-->}[d]&(B, \beta)\ar[d]^{\xi_B}\ar[r]&(A, \alpha)\ar[d]^{\xi_A}\\
    (SA', S\alpha')\ar[r]&(C', \gamma')\ar[r]&(B', \beta')\ar[r]&(A', \alpha'),
    }\]
    there exists a dotted arrow $\xi_C\in KK^G((C, \gamma), (C', \gamma'))$ making the diagram commute.
    Furthermore,
    if $\xi_A, \xi_B$ are $KK^G$-equivalences, then so is $\xi_C$.
\end{lem}

By \cite[19.2.6.]{B} (cf \cite[Sec. 4.1, Remark 4.1, Appendix]{ps}),
there exists a natural isomorphism $\eta_{A, B} : KK(A, B)\to \operatorname{Ext}^1(A, SB\otimes \K)$ satisfying
\[\eta_{A, A}(I_{A})=i_{(1_{\M(A)}\otimes e)}^*([SA\otimes\K\to C_0(0, 1]\otimes A\otimes\K\to A\otimes\K])\in\operatorname{Ext}^1(A, SA\otimes\K),\]
where $e\in\K$ is a minimal projection and $i_{(1_{\M(A)}\otimes e)} : A\ni a\mapsto a\otimes e\in A\otimes\K$ is the corner embedding of the projection $1_{\M(A)}\otimes e\in\M(A\otimes \K)$.
\begin{lem}[{cf \cite[Sec. 2.3]{MN}}]\label{ade}
Let $C$ be a unital separable nuclear C*-algebra, and $B$ be a separable nuclear C*-algebra.
Recall the embedding $i_{(1_{\M(B)}\otimes e)} : B\ni x\mapsto x\otimes e\in B\otimes\K$.
    For any $\xi\in KK(SC, B)$, there exists a non-unital essential extension $B\otimes \K\to E\xrightarrow{\pi_E}C$ such that
    \[SC\xrightarrow{\xi\hat{\otimes}KK(i_{(1_{\M(B)}\otimes e)})}B\otimes \K\to E\xrightarrow{\pi_E}C\]
    is an exact triangle.
    Furthermore, if $B=\mathbb{C}$, the extension $E$ is chosen to be unital (i.e., $1_{\M(\K)}=1_E\in E$).
\end{lem}
\begin{proof}
By the suspension isomorphism $\operatorname{Ext}^1(C, B\otimes\K)\cong \operatorname{Ext}^1(SC, SB\otimes\K)$,
there exists a non-unital essential extension $B\otimes\K\to E\xrightarrow{\pi_E} C$ satisfying
\[\eta_{SC, B}(\xi)=[SB\otimes \K\to SE\xrightarrow{S\pi_E}SC]\in \operatorname{Ext}^1(SC, SB\otimes\K).\]
If $B=\mathbb{C}$, \cite[Theorem 2.3.]{Ska} (cf \cite[Theorem 4.4.]{ps}) allows us to assume that $E$ is unital.
For general $B$, we may assume $E$ is non-unital by adding an appropriate trivial extension to the Busby invariant of $E$.

For the exact triangle
\[SC\xrightarrow{\iota(E)}\operatorname{Cone}(\pi_E)\xrightarrow{ev_1} E\xrightarrow{\pi_E}C,\]
we consider the following commutative diagram
\[\xymatrix{
KK(SC, B)\ar[r]^{\eta_{SC, B}}&\operatorname{Ext}^1(SC, SB\otimes\K)\\
KK(\operatorname{Cone}(\pi_E), B)\ar[u]^{\iota(E)^*}\ar@<+0.8 ex>[r]^{\eta_{\operatorname{Cone}, B}}\ar[d]^{j(E)^*}&\operatorname{Ext}^1(\operatorname{Cone}(\pi_E), SB\otimes\K)\ar[u]^{\iota(E)^*}\ar[d]^{j(E)^*}\\
KK(B\otimes\K, B)\ar[d]^{i_{(1_{\M(B)}\otimes e)}^*}\ar[r]^{\eta_{B\otimes \K, B}}&\operatorname{Ext}^1(B\otimes\K, SB\otimes\K)\ar[d]^{i_{(1_{\M(B)}\otimes e)}^*}\\
KK(B, B)\ar[r]^{\eta_{B, B}}&\operatorname{Ext}^1(B, SB\otimes\K).
}\]
Since
\[\iota(E)^*([SB\otimes\K\to C_0(0, 1]\otimes E\to \operatorname{Cone}(\pi_E)])=[SB\otimes\K\to SE\to SC],\]
\[j(E)^*([SB\otimes\K\to C_0(0, 1]\otimes E\to \operatorname{Cone}(\pi_E)])=[SB\otimes\K\to C_0(0, 1]\otimes B\otimes\K\to B\otimes\K],\]
\[\eta_{B, B}(I_B)=i_{(1_{\M(B)}\otimes e)}^*([SB\otimes\K\to C_0(0, 1]\otimes B\otimes\K\to B\otimes\K]),\]
diagram chasing shows
\begin{align*}
    \xi\hat{\otimes}KK(i_{(1_{\M(B)}\otimes e)})=&\iota(E)^*({j(E)^*}^{-1}({i_{(1_{\M(B)}\otimes e)}^*}^{-1}(I_B)))\hat{\otimes}KK(i_{(1_{\M(B)}\otimes e)})\\
    =&KK(\iota(E))\hat{\otimes}\left(({j(E)^*}^{-1}({i_{(1_{\M(B)}\otimes e)}^*}^{-1}(I_B)))\hat{\otimes} KK(i_{(1_{\M(B)}\otimes e)})\right),
\end{align*}
and
\begin{align*}&KK(j(E))\hat{\otimes}({j(E)^*}^{-1}({i_{(1_{\M(B)}\otimes e)}^*}^{-1}(I_B)))\hat{\otimes} KK(i_{(1_{\M(B)}\otimes e)})\\
=&KK(j(E))\hat{\otimes} KK(j(E))^{-1}\hat{\otimes}KK(i_{(1_{\M(B)}\otimes e)})^{-1}\hat{\otimes}KK(i_{(1_{\M(B)}\otimes e)})\\
=&I_{B\otimes\K}.\end{align*}
This implies \[KK(j(E))^{-1}=({j(E)^*}^{-1}({i_{(1_{\M(B)}\otimes e)}^*}^{-1}(I_B)))\hat{\otimes} KK(i_{(1_{\M(B)}\otimes e)}),\]
\[\xi\hat{\otimes} KK(i_{(1_{\M(B)}\otimes e)})=KK(\iota(E))\hat{\otimes}KK(j(E))^{-1},\]
and we have the following isomorphism of triangles
\[\xymatrix{
SC\ar[r]_{\xi\hat{\otimes} KK(i_{1\otimes e})}\ar@{=}[d]&B\otimes\K\ar[d]^{KK(j(E))}\ar[r]&E\ar@{=}[d]\ar[r]&C\ar@{=}[d]\\
SC\ar[r]_{KK(\iota(E))}&\operatorname{Cone}(\pi_E)\ar[r]&E\ar[r]&C.
}\]
\end{proof}
Note that, for general $B$, every element in $\operatorname{Ext}^1(A, B\otimes \K)$ is not necessary given by the unital extension because of the obstruction $K_1(B)$ (see the 6-term exact sequence in \cite{Ska}, \cite[Theorem 4.4.]{ps}).


\section{Idea of Construction}\label{keycon}
In this section,
we explain our construction in the simplest setting as in the following theorem.
\begin{thm}\label{Mat0}
    For a unital Kirchberg algebra $A$, there exists an aperiodic automorphism $\gamma\in \operatorname{Aut}(A)$ satisfying $A^\gamma=\C$.
\end{thm}
KK-theoretic data of $A$ is given by the exact triangle
\[SA\to \operatorname{Cone}(1_A)\xrightarrow{ev_1}\C\xrightarrow{1_A}A.\]
Since $A$ is separable nuclear,
the mapping cone $\operatorname{Cone}(1_A)$ is also separable nuclear,
and \cite[Lemma 2.2.]{fkk} shows that there exist a separable nuclear unital C*-algebra $C$ and an embedding $f : S\operatorname{Cone}(1_A)\hookrightarrow C$ providing a KK-equivalence $KK(f)\in KK(S\operatorname{Cone}(1_A), C)^{-1}$.
The element
\[\xi:=KK(Sf)^{-1}\hat{\otimes} (b^{-1}\otimes I_{\operatorname{Cone}(1_A)})\hat{\otimes} KK(ev_1)\hat{\otimes} KK(i_e)\in KK(SC, \K)\]
makes the following diagram commute
\[\xymatrix{
\operatorname{Cone}(1_A)\ar[r]^{KK(ev_1)}\ar[d]^{b\otimes I_{\operatorname{Cone}(1_A)}}&\C\ar[d]^{KK(i_e)}\\
S(S\operatorname{Cone}(1_A))\ar[d]^{KK(Sf)}&\K\\
SC\ar@{-->}[ur]_{\xi}&
},\]
where $b\in KK(\C, S^2)^{-1}$ is the Bott element.
By Lemma \ref{ade},
we have the following commutative diagram of the exact triangles
\[\xymatrix{
S\C\ar@<+0.8 ex>[r]^{-KK(S1_A)}&SA\ar@<+0.8 ex>[r]&\operatorname{Cone}(1_A)\ar@<+0.8 ex>[r]^{\quad\quad KK(ev_1)}\ar[d]^{\sim_{KK}}&\C\ar[d]^{KK(i_{e})}&&\\
&&SC\ar@{-->}[r]^{\xi}&\K\ar[r]^{l_E}&E\ar[r]^{\pi_E}&C,
}\]
where $E$ is a unital essential extension.
Fix a non-degenerate representation $E\subset \mathbb{B}(H)$.
Let $X:=H\otimes_\C l^2(\mathbb{Z})\otimes_\C E$ be a Hilbert $E-E$-bimodule (see \cite{Kumjian} and note that $E\otimes 1_{l^2(\mathbb{Z})}\cap\K(H\otimes l^2(\mathbb{Z}))=\{0\}$).
Let $\mathcal{T}_X$ be the Toeplitz--Pimsner algebra generated by the creation operators
\[T_{\zeta\otimes\delta_n\otimes x}, \quad \zeta\in H,\;\; n\in\mathbb{Z},\;\; x\in E.\]
By \cite[Theorem 3.1.]{Kumjian},
$\mathcal{T}_X$ is a unital Kirchberg algebra for which the natural embedding $E\hookrightarrow \mathcal{T}_X$ is a KK-equivalence.
Applying Lemma \ref{mori}, Lemma \ref{benri} and Theorem \ref{KP} for the diagram
\[\xymatrix{
S\C\ar[r]^{-KK(S1_A)}\ar[d]^{KK(Si_e)}&SA\ar[r]&\operatorname{Cone}(1_A)\ar[d]^{\sim_{KK}}\ar[r]&\C\ar[d]^{KK(i_{e})}\\
S\K\ar[r]^{-KK(Sl_E)}&SE\ar@<+0.8 ex>@{^{(}->}[d]^{\sim_{KK}}\ar[r]^{-KK(S\pi_E)}&SC\ar@{-->}[r]^{\xi}&\K\\
S\C\ar[u]^{KK(Si_e)}\ar@<+0.8 ex>[r]^{-KK(S\C e\subset S\mathcal{T}_X)\;\;}&S\mathcal{T}_X&&\\
S\C\ar@{=}[u]\ar[r]_{-KK(1_{e(\mathcal{T}_X)e})}&Se(\mathcal{T}_X)e\ar[u]^{KK(Si_e)}&&
}\]
we have $A\cong e(\mathcal{T}_X)e$.
\begin{lem}\label{Mats2}
    Let $\Gamma\in \operatorname{Aut}(\mathcal{T}_X)$ be a quasi-free automorphism given by
    \[\Gamma (T_{\zeta\otimes\delta_n\otimes x}):=T_{\zeta\otimes \delta_{n+1}\otimes x},\quad \text{for}\;\; \zeta\in H,\; n\in\mathbb{Z},\; x\in E.\]
    Then, we have $(\mathcal{T}_X)^\Gamma=(\mathcal{T}_X)^{\Gamma^N}=E$ for $N\in\mathbb{Z}\backslash\{0\}$.
\end{lem}
\begin{proof}
Since $\Gamma^N(y)=\Gamma^N(T_{\zeta\otimes\delta_{0}\otimes 1_E}^*T_{\zeta\otimes \delta_{0}\otimes y})=T_{\zeta\otimes\delta_{N}\otimes 1_E}^*T_{\zeta\otimes \delta_{N}\otimes y}=y$ for $y\in E\subset\mathcal{T}_X, \zeta\in H, ||\zeta||=1$,
it is enough to show $(\mathcal{T}_X)^{\Gamma^N}\subset E$.

    Fix $Y\in (\mathcal{T}_X)^{\Gamma^N}$ and $\epsilon>0$.
    There exist a finite subset $F_1\subset \mathbb{Z}$ and finitely many elements $y\in E, \;\mu_i, \nu_i, \xi_i, \eta_i\in \bigcup_{k=1}^\infty (H\otimes l^2(F_1)\otimes E)^{\otimes_E k}$ satisfying
    \[||Y-(y+\sum_i T_{\mu_i}+T_{\nu_i}^*+T_{\xi_i}T_{\eta_i}^*)||<\epsilon.\]
    For the Fock representation $\mathcal{T}_X\subset \mathcal{L}_E(E\Omega_X\oplus\bigoplus_{k=1}^\infty X^{\otimes_E k})$,
    there exist a finite set $F_2\subset\mathbb{Z}$ \[u=u_0\Omega_X+u_1, v=v_0\Omega_X+v_1\in E\Omega_X\oplus \bigoplus_{k=1}^\infty (H\otimes l^2(F_2)\otimes E)^{\otimes_E k}\]satisfying
    \[||u||, ||v||\leq 1, \quad ||Y-y||\approx_\epsilon ||\langle(Y-y)u, v\rangle||,\]
where $a\approx_\epsilon b$ means $|a-b|<\epsilon$.
    Since $F_1, F_2$ are finite, there exists $m\in\mathbb{Z}$ such that $(mN+F_1)\cap F_2=\emptyset$ and one has
    \[\Gamma^{mN}(T_{\mu_i}^*)v_1=\Gamma^{mN}(T_{\nu_i}^*)u_1=\Gamma^{mN}(T_{\xi_i}T_{\eta_i}^*)u_1=0.\]
    Since $\mu_i, \nu_i, \eta_i\in \bigcup_{k\geq1} X^{\otimes_E k}$,
    one has $\Gamma^{mN}(T_{\mu_i}^*)v_0\Omega_X=\Gamma^{mN}(T_{\nu_i}^*)u_0\Omega_X=\Gamma^{mN}(T_{\xi_i}T_{\eta_i}^*)u_0\Omega_X=0$ so that 
    \[\Gamma^{mN}(T_{\mu_i}^*)v=\Gamma^{mN}(T_{\nu_i}^*)u=\Gamma^{mN}(T_{\xi_i}T_{\eta_i}^*)u=0.\]
    
    Direct computation yields
    \begin{align*}
        ||Y-y||&\approx_\epsilon ||\langle(Y-y) u, v\rangle||\\
        &=||\langle\Gamma^{mN}(Y-y) u, v\rangle||\\
        &\approx_\epsilon ||\langle(\sum_i\Gamma^{mN}(T_{\mu_i}+T^*_{\nu_i}+T_{\xi_i}T_{\eta_i}^*))u, v\rangle||=0.
    \end{align*}
    This implies $Y\in E$ (i.e., $(\mathcal{T}_X)^{\Gamma^{N}}=E$).
\end{proof}
We write $\gamma:=\Gamma |_{e(\mathcal{T}_X)e}\in\operatorname{Aut}(A)$.
\begin{lem}\label{Mats3}
    For $N\in\mathbb{Z}\backslash \{0\}$,
    $\gamma^N$ is outer.
\end{lem}
\begin{proof}
    Assume that there exists $N$ and $V\in U(e(\mathcal{T}_X)e)$ with $\gamma^N=\ad V$.
    By Lemma \ref{Mats2} and $\Gamma^N(V)=VVV^*=V$,
    one has $V\in (\mathcal{T}_X)^{\Gamma^N}=E$.
    Since $V\in e(\mathcal{T}_X)e$,
    we have $V\in eEe=\C e$ (i.e., $\gamma^N=\id$).
    This is a contradiction because we have
    \[\gamma^N (T_{e\zeta\otimes \delta_0\otimes e})=eT_{\zeta\otimes\delta_N\otimes e}\not=T_{e\zeta\otimes\delta_0\otimes e}\in e(\mathcal{T}_X)e.\]
    Thus, $\gamma^N$ is outer.
\end{proof}
\begin{proof}[{Proof of Theorem \ref{Mat0}}]
    By Lemma \ref{Mats2},
    we have $A^\gamma=(e(\mathcal{T}_X)e)^\Gamma\subset eEe=\C e$,
    and Lemma \ref{Mats3} shows that $\gamma$ is aperiodic.
\end{proof}
Our proofs of Theorem \ref{Mats} and Theorem \ref{fee} in the following two sections are given by technical variants of the above argument.

\section{Proof of Theorem \ref{Mats}}
\subsection{Pimsner's Construction for Theorem \ref{Mats}}
Throughout this subsection,
we assume the following:
\begin{enumerate}
    \item $B$ is unital, separable, nuclear,
    \item $E$ is separable, non-unital, nuclear, and has an approximate unit consisting of increasing sequence of projections $\{p_n\}_{n=1}^\infty\subset E,\;\; p_n<p_{n+1}$,
    \item $B\otimes\K\triangleleft E$ is an essential ideal (i.e., $E\subset\M(B\otimes\K)$),
    \item There is a non-degenerate representation $E\subset \mathbb{B}(H)$ and a unitary representation $\{U_g\}_{g\in G}\subset U(\mathbb{B}(H))$ of a countable infinite discrete group $G$ satisfying $U_g E U_g^*=E$, $U_g (B\otimes\K)U_g^*=B\otimes\K$ and $U_g(1_B\otimes e_0)=(1_B\otimes e_0)U_g=1_B\otimes e_0$, where $e_0\in\K$ is a minimal projection. 
\end{enumerate}
We will choose the trivial representation $U_g=1$ later.
Note $\operatorname{\ad}U_g\in\operatorname{Aut}(E)$ and we frequently write $\operatorname{\ad}U_g(x)=g\cdot x,\; x\in E$ for short.
Let $X$ be the following Hilbert $E-E$-bimodule:
\[X:=H\otimes_\mathbb{C} l^2(G)\otimes_\mathbb{C} E,\quad a(\zeta\otimes\delta_h\otimes x)b:=a(\zeta)\otimes \delta_h\otimes (xb),\;\; a, b, x\in E,\;\; \zeta\in H,\;\; h\in G.\]
A $G$-action on $X$ is given by
\[g\cdot(\zeta\otimes\delta_h\otimes x):=U_g(\zeta)\otimes\delta_{gh}\otimes (g\cdot x)\]
and the following holds
\[g\cdot ((\zeta\otimes\delta_h\otimes x)b)=U_g(\zeta)\otimes\delta_{gh}\otimes g\cdot(xb)=g\cdot(\zeta\otimes\delta_h\otimes x)(g\cdot b),\]
\[g\cdot (a(\zeta\otimes\delta_h\otimes x))=U_gaU_g^*U_g(\zeta)\otimes\delta_{gh}\otimes (g\cdot x)=(g\cdot a)(g\cdot (\zeta\otimes\delta_h\otimes x)),\]
\[\langle g\cdot(\zeta\otimes\delta_h\otimes x), g\cdot(\xi\otimes\delta_k\otimes y)\rangle_X=\langle U_g(\zeta), U_g(\xi)\rangle_H\delta_{h, k}(g\cdot x)^*(g\cdot y)=g\cdot (\langle \zeta\otimes\delta_h\otimes x, \xi\otimes\delta_k\otimes y\rangle_X)\]
(see \cite{Mey, P}).

From this $G$-Hilbert $E-E$-bimodule $X$ and the Fock space
\[\mathcal{F}(X):=E\Omega_X\oplus\bigoplus_{k=1}^\infty X^{\otimes_Ek},\]
the Toeplitz--Pimsner algebra is given by
\[\mathcal{T}_X:=C^*(\{T_{\zeta\otimes\delta_h\otimes x}, E\;|\; \zeta\in H,\; h\in G,\; x\in E\})=C^*(\{T_{\zeta\otimes\delta_h\otimes x}\;|\; \zeta\in H,\; h\in G,\; x\in E\})\subset \mathcal{L}_E(\mathcal{F}(X))\]
because of $\overline{\langle X, X\rangle_X}^{||\cdot||}=E$.
A quasi-free action $\Gamma$ of $G$ on $\mathcal{T}_X$ is defined by
\[\Gamma_g( T_{(\zeta\otimes\delta_h\otimes x)}):=T_{g\cdot(\zeta\otimes\delta_h\otimes x)}.\]
Note that $\overline{\langle X, X\rangle_X}^{||\cdot||}=E$ implies
\[E=\overline{\operatorname{span}}\{T_\mu^*T_\nu\; |\; \mu, \nu\in X^{\otimes_E^k},\;\; k\in\mathbb{N}\}\subset\mathcal{T}_X,\]
and this embedding is $G$-equivariant.
\begin{thm}[{\cite[Proof of Thm. 4.4, Rem. 4.10.]{P}}]\label{wk}
    The embedding $(E, \operatorname{\ad}U_\cdot)\subset (\mathcal{T}_X, \Gamma)$ is a $KK^G$-equivalence.
\end{thm}

\begin{lem}\label{kir}
$\mathcal{T}_X$ is a stable Kirchberg algebra.
\end{lem}
\begin{proof}
By our assumption 2.,
we take an increasing sequence of projections $p_n<p_{n+1}\in E$ which gives an approximate unit.
Recall $|G|=\infty$ (i.e., $p_n E p_n\otimes 1_{l^2(G)}\cap \K(p_nH\otimes l^2(G))=\{0\}$).
Applying \cite{Kumjian} for a Hilbert $p_nEp_n-p_nEp_n$-bimodule $X_n :=p_nH\otimes l^2(G)\otimes p_nEp_n$,
the Toeplitz--Pimsner algebra $\mathcal{T}_{X_n}$ is a unital Kirchberg algebra and the universality of Toeplitz--Pimsner algebra gives an inclusion
\[\mathcal{T}_{X_n}\ni T_\zeta\mapsto T_\zeta\in \mathcal{T}_X, \quad \zeta\in X_n.\]
Since the representation $E\subset \mathbb{B}(H)$ is non-degenerate,
one has $\overline{\bigcup_{n=1}^\infty p_n H}=H$ and $\mathcal{T}_X=\overline{\bigcup_{n=1}^\infty\mathcal{T}_{X_n}}$.
Note that  $\mathcal{T}_X$ is non-unital because $\{1_{\mathcal{T}_{X_n}}\}_{n=1}^\infty=\{p_n\}_{n=1}^\infty$ is an approximate unit of $\mathcal{T}_X$ satisfying $||p_n-p_{n+1}||=1$. 
Thus, the permanence properties and Zhang's dichotomy theorem (see \cite{Zha}, \cite[Prop. 4.1.3, Prop. 4.1.8]{Ro}) shows that $\mathcal{T}_X$ is a stable Kirchberg algebra.
\end{proof}

Recall $\Gamma_g(1_B\otimes e_0)=U_g(1_B\otimes e_0)U_g^*=1_B\otimes e_0=U_g(1_B\otimes e_0)$.
The corner \[(A_E, \gamma^E):=({_{(1_B\otimes e_0)}{(\mathcal{T}_X)}_{(1_B\otimes e_0)}}, \Gamma |)\] is a unital Kirchberg $G$-algebra with the unital embedding $\iota : B\ni b\mapsto b\otimes e_0 \in (1_B\otimes e_0)E(1_B\otimes e_0)\subset A_E$ making the following diagram commute
\[\xymatrix{
(B, \id)\ar[d]^{\iota}\ar@{=}[r]&({_{(1_B\otimes e_0)}(B\otimes\K)_{(1_B\otimes e_0)}}, \ad U_\cdot)\ar[d]\ar[r]^{\quad\quad\quad\; Lem \ref{mori}}&(B\otimes\K, \ad U_\cdot)\ar[d]\ar@{=}[r]&(B\otimes\K, \ad U_\cdot)\ar[d]\\
(A_E, \gamma^E)\ar@{=}[r]&({_{(1_B\otimes e_0)}(\mathcal{T}_X)_{(1_B\otimes e_0)}}, \Gamma|)\ar[r]^{\quad\quad\quad Lem \ref{mori}, \;\ref{kir} }&(\mathcal{T}_X, \Gamma)&(E, \ad U_\cdot)\ar[l]^{Thm \ref{wk}}
}\]
where all maps are equivariant $*$-homomorphisms and all horizontal arrows are $KK^G$-equivalences.
\begin{cor}\label{yhe}
    For an equivariant extension satisfying the assumptions 1. 2. 3. 4.,
    there are $KK^G$-equivalences $k_E\in KK^G((A_E, \gamma^E), (E, \ad U_\cdot))^{-1}$ and $KK^G(i_{(1_B\otimes e_0)})\in KK^G((B, \id), (B\otimes\K, \ad U_\cdot))^{-1}$ making the following diagram commute
    \[
    \xymatrix{(B, \id)\ar[d]^{KK^G(\iota)}\ar@<+0.8 ex>[r]^{KK^G(i_{(1_B\otimes e_0)})\quad\quad\quad}&(B\otimes\K, \ad U_\cdot)\ar[d]\\
    (A_E, \gamma^E)\ar[r]^{k_E}&(E, \ad U_\cdot).
    }\]  
\end{cor}
\begin{lem}[{cf \cite[Thm. 2.1.]{Mey}}]\label{out}
    For $g\not=e$, the automorphism $\gamma^E_g$ is outer.
\end{lem}
\begin{proof}
    Assume that there is a unitary $V\in {_{(1_B\otimes e_0)}(\mathcal{T}_X)_{(1_B\otimes e_0)}}$ satisfying $\ad V=\gamma^E_g$.
    First, we show that $\beta_z(V)=V$ for the gauge action
    \[\beta_z : \mathcal{T}_X\ni T_{\zeta\otimes\delta_h\otimes x}\mapsto zT_{\zeta\otimes\delta_h\otimes x}\in\mathcal{T}_X,\; z\in\mathbb{T}=\{z\in\C\;|\; |z|=1\}.\]
    Note that $\beta_z$ restricts to ${_{(1_B\otimes e_0)}(\mathcal{T}_X)_{(1_B\otimes e_0)}}$ because $\beta_z(1_{B}\otimes e_0)=1_{B}\otimes e_0$ and the ristriction of $\beta_z$ is implemented by the following diagonal unitary
    \[W(z):=\left(\begin{array}{ccccc}
1&0&0&0\cdots\\
0&z&0&0\cdots\\
0&0&z^2&0\cdots\\
0&0&0&z^3\cdots\\
\vdots&\vdots&\vdots&\vdots\ddots
\end{array}\right)\in \mathcal{L}_E({_{(1_B\otimes e_0)}E}\Omega_X\oplus\bigoplus_{k=1}^\infty{_{(1_B\otimes e_0)}X}\otimes_EX^{\otimes_E k-1}).\]
Since $\beta_z\circ\Gamma_g(T_{\zeta\otimes\delta_h\otimes x})=T_{z g\cdot(\zeta\otimes\delta_h\otimes x)}=\Gamma_g\circ\beta_z(T_{\zeta\otimes\delta_h\otimes x})$,
one has $\Gamma_g\circ \beta_z=\beta_z\circ\Gamma_g\in\operatorname{Aut}(\mathcal{T}_X)$,
and the element $\beta_z(V)V^*$ lies in the center of the unital Kirchberg algebra ${_{(1_B\otimes e_0)}(\mathcal{T}_X)_{(1_B\otimes e_0)}}$ (i.e., $\beta_z(V)V^*\in \mathbb{T}$).
Since $\beta_z(V)V^*\beta_w(V)V^*=\beta_w((\beta_z(V)V^*)V)V^*=\beta_{zw}(V)V^*$,
the map $\mathbb{T}\ni z\mapsto \beta_z(V)V^*\in\mathbb{T}$ is a continuous group homomorphism and there is $m\in\mathbb{Z}$ with $\beta_z(V)V^*=z^m$.
We check $m=0$.
In the case $m>0$,
one has
\begin{align*}
    \langle y\Omega_X, V\eta\rangle_{\mathcal{F}(X)}=&\langle W(\bar{z})y\Omega_X, V\eta\rangle_{\mathcal{F}(X)}\\
    =&\langle y\Omega_X, W(z)V\eta\rangle_{\mathcal{F}(X)}\\
    =&\langle y\Omega_X, VW(z)\eta\rangle_{\mathcal{F}(X)}z^m\\
    =&\langle y\Omega_X, V\eta\rangle_{\mathcal{F}(X)}z^{m+|\eta|}
\end{align*}
for $y\in (1_B\otimes e_0)E$ and $\eta\in (1_B\otimes e_0)X\otimes_EX^{\otimes_E|\eta|-1}$, $\eta\in (1_B\otimes e_0)E\Omega_X$ ($|\eta|=0$),
and this implies $\operatorname{Im}V\perp (1_B\otimes e_0)E\Omega$.
This is a contradiction because $V$ is a unitary of ${_{(1_B\otimes e_0)}(\mathcal{T}_X)_{(1_B\otimes e_0)}}\subset\mathcal{L}_E({_{(1_B\otimes e_0)}E}\Omega_X\oplus\bigoplus_{k=1}^\infty{_{(1_B\otimes e_0)}X}\otimes_EX^{\otimes_E k-1})$, and one has $m\leq 0$.
Now the same argument for $V^*$ shows $m=0$ (i.e., $V$ is $\beta$-invariant).

    For any $\epsilon>0$ with $2\epsilon+\epsilon^2<1/2$,
    there exist a finite set $F\subset G$ and finitely many elements $a\in E$, $\mu_i, \nu_i, \xi_i, \eta_i\in \bigcup_{k=1}^\infty (H\otimes l^2(F)\otimes E)^{\otimes_E k}$ approximating $V$ as follows:
    \[||V-(1_B\otimes e_0)\left(a+\sum_i (T_{\mu_i}+T_{\nu_i}^*+T_{\xi_i}T_{\eta_i}^*)\right)(1_B\otimes e_0)||<\epsilon.\]
Applying $\int_\mathbb{T}\beta_z(\cdot) dz$ to the above,
we may assume that
\[||V-(1_B\otimes e_0)\left(a+\sum_iT_{\xi_i}T_{\eta_i}^*\right)(1_B\otimes e_0)||<\epsilon, \quad a\in B\otimes e_0,\quad  |\xi_i|=|\eta_i|\geq 1.\]
Since $|G|=\infty$,
there is $h\in G\backslash F$.
    Fix a vector $\zeta\in (1_B\otimes e_0)H,\; ||\zeta||=1$.
    Since $H\otimes \delta_h\otimes E\perp H\otimes l^2(F)\otimes E$,
    the direct computation yields
    \begin{align*}
    \zeta\otimes\delta_{gh}\otimes (1_B\otimes e_0)
    &=U_g(1_B\otimes e_0)(\zeta)\otimes \delta_{gh}\otimes g\cdot (1_B\otimes e_0)\\
    &=\gamma_g(T_{\zeta\otimes\delta_h\otimes(1_B\otimes e_0)})\Omega_X\\
        &=VT_{\zeta\otimes\delta_h\otimes (1_B\otimes e_0)}V^*\Omega_X\\
        &\approx_{\epsilon+\epsilon(1+\epsilon), ||\cdot||}a(\zeta)\otimes\delta_h\otimes a^*+(1_B\otimes e_0)\sum_iT_{\xi_i}T_{\eta_i}^*(\zeta\otimes\delta_h\otimes a^*)\\
        &=a(\zeta)\otimes\delta_h\otimes a^*+0,
    \end{align*}
    and $g\not=e$ yields the  following contradiction
    \begin{align*}
        1&=||\langle \zeta\otimes\delta_{gh}\otimes(1_B\otimes e_0), \zeta\otimes\delta_{gh}\otimes (1_B\otimes e_0)\rangle_X||\\
        &\approx_{2\epsilon+\epsilon^2}||\langle a(\zeta)\otimes\delta_h\otimes a, \zeta\otimes\delta_{gh}\otimes (1_B\otimes e_0)\rangle_X||=0.\\
    \end{align*}
    Thus, $\gamma^E_g$ must be outer.
\end{proof}
Recall the embedding $\iota : B\ni b\mapsto b\otimes e_0\in B\otimes e_0\subset A_E$.
\begin{thm}\label{km}
    We have $A_E^{\gamma^E}=\iota(B)$ and there is a $G$-equivariant conditional expectation \[\langle-\Omega_X, \Omega_X\rangle_{\mathcal{F}(X)} : A_E\to A_E^{\gamma^E}.\]
\end{thm}
\begin{proof}
The $G$-equivariant conditional expectation $\langle-\Omega_X, \Omega_X\rangle_{\mathcal{F}(X)} : \mathcal{T}_X\to E$ restricts to $A_E\to \iota(B)$.
We will show $A_E^{\gamma^E}=\iota(B)$.
Since $U_g (1_B\otimes e_0)=1_B\otimes e_0$,
one has $\iota(B)=B\otimes e_0\subset A_E^{\gamma^E}$.
    Fix $Z\in A_E^{\gamma^E}$ and $\epsilon>0$.
    There exist a finite set $F_1\subset G$ and finitely many elements $a\in E, \; \mu_i,\; \nu_i,\; \xi_i,\;\eta_i\in \bigcup_{k=1}^\infty(H\otimes l^2(F_1)\otimes E)^{\otimes_E k}$ approximating $Z$ as follows:
    \[||Z-(1_B\otimes e_0)(a+\sum_iT_{\mu_i}+T^*_{\nu_i}+T_{\xi_i}T_{\eta_i}^*)(1_B\otimes e_0)||<\epsilon.\]
We write $Y:=(1_B\otimes e_0)a(1_B\otimes e_0)\in \iota(B)\subset A_E^{\gamma^E}$.
By the definition of operator norm,
there exist a finite set $F_2\subset G$ and two elements $u, v\in \mathcal{F}(H\otimes l^2(F_2)\otimes E)\subset \mathcal{F}(X)$ with $||u||_{\mathcal{F}(X)}, ||v||_{\mathcal{F}(X)}\leq 1$ satisfying
\[||Z-Y||\approx_\epsilon ||\langle (Z-Y) u, v\rangle_{\mathcal{F}(X)}||.\]
We may assume 
\[u=u_0+u_1,\; u_0\in (1_B\otimes e_0)E\Omega_X,\; u_1\in \bigoplus_{k=1}^\infty ((1_B\otimes e_0)H\otimes l^2(F_2)\otimes E)^{\otimes_E k},\]
\[v=v_0+v_1,\; v_0\in (1_B\otimes e_0)E\Omega_X,\; v_1\in \bigoplus_{k=1}^\infty ((1_B\otimes e_0)H\otimes l^2(F_2)\otimes E)^{\otimes_E k}.\]
Since $|G|=\infty$,
there is $g\in G$ satisfying $g F_1\cap F_2=\emptyset$
and one has
\[\Gamma_g(T_{\mu_i}^*(1_B\otimes e_0))(v)=\Gamma_g(T_{\mu_i}^*)(v)=0+\Gamma_g(T_{\mu_i}^*)(v_1)=0,\]
\[\Gamma_g(T^*_{\nu_i}(1_B\otimes e_0))(u)=\Gamma_g(T_{\nu_i}^*)(u)=0+\Gamma_g(T^*_{\nu_i})(u_1)=0,\;\; \Gamma_g(T^*_{\eta_i}(1_B\otimes e_0))(u)=0,\]
because $H\otimes l^2(gF_1)\otimes E\perp (1_B\otimes e_0)H\otimes l^2(F_2)\otimes E$.
Thus, the direct computation yields
\begin{align*}
    ||Z-Y||&\approx_\epsilon||\langle(Z-Y)u, v\rangle_{\mathcal{F}(X)}||\\
    &=||\langle \Gamma_g(Z-Y) u, v\rangle_{\mathcal{F}(X)}||\\
    &\approx_{\epsilon ||u||||v||}||\langle\Gamma_g((1_B\otimes e_0)(\sum_iT_{\mu_i}+T^*_{\nu_i}+T_{\xi_i}T_{\eta_i}^*)(1_B\otimes e_0))u, v\rangle_{\mathcal{F}(X)}||\\
    &=0.
\end{align*}
For any $\epsilon>0$,
there exists $Y_\epsilon\in\iota(B)$ with $||Z-Y_\epsilon||<2\epsilon$ by the above argument,
and this shows $Z\in \iota(B)$ (i.e., $A_E^{\gamma^E}\subset \iota(B)$).
\end{proof}

\subsection{Proof of Theorem \ref{Mats}}
In this section,
we assume that $A$ is a unital Kirchberg algebra and $B$ is a unital, separable, nuclear C*-algebra.

For a unital embedding $\iota_0 : B\hookrightarrow A$,
we will find an extension $B\otimes\K\triangleleft E$ and a KK-equivalence $\eta\in KK(A, E)^{-1}$ making the following diagram commute
\[\xymatrix{B\otimes\K\ar[r]&E\\
B\ar[u]^{KK(i_{(1_B\otimes e_0)})}\ar[r]^{KK(\iota_0)}&A.\ar[u]^{\eta}}\]
Recall that $\operatorname{Cone}(\iota_0):=\{(a(t), b)\in (C_0(0, 1]\otimes A) \oplus B\;|\; a(1)\in\iota_0(b)\}$ denotes the mapping cone of $\iota_0$ (see \cite[Appendix A]{MN}), and  $ev_1 : \operatorname{Cone}(\iota_0)\ni (a(t), b)\mapsto b=\iota_0^{-1}(a(1))\in B$ denotes the evaluation at $t=1$.
By \cite[Lem. 2.2.]{fkk},
there exists a unital separable, nuclear C*-algebra $C$ with an embedding $f : S\operatorname{Cone}(\iota_0)\hookrightarrow C$ satisfying
\[KK(f)\in KK(S\operatorname{Cone}(\iota_0), C)^{-1},\quad [1_C]_0=0\in K_0(C).\]
We obtain an element $\varphi\in KK(SC, B\otimes \mathbb{K})$ from the following commutative diagram (i.e., $\varphi:=KK(Sf)^{-1}\hat{\otimes}(b^{-1}\otimes I_{\operatorname{Cone}(\iota_0)})\hat{\otimes}KK((i_{1_B\otimes e_0})\circ ev_1)$)
\[\xymatrix{
\operatorname{Cone}(\iota_0)\ar[r]^{KK(ev_1)}\ar[d]^{b\otimes I_{\operatorname{Cone}(\iota_0)}}&B\ar[d]^{KK(i_{1_B\otimes e_0})}\\
SS\operatorname{Cone}(\iota_0)\ar[d]^{KK(Sf)}&B\otimes\mathbb{K}\\
SC,\ar[ur]_{\varphi}&
}
\]
where $b\in KK(\mathbb{C}, S^2)^{-1}$ is the Bott element.
\begin{lem}\label{extri}
    There exists a non-unital, essential extension $B\otimes\K\triangleleft E\subset\M(B\otimes\K)$ with the following exact triangle:
    \[SC\xrightarrow{\varphi}B\otimes\K\xrightarrow{l_E} E\xrightarrow{(\tau^{-1}\circ q_{B})}C,\]
    where $\tau$ is the Busby invariant as in the  diagram below
    \[\xymatrix{
    E\ar[d]\ar[r]^{\tau^{-1}\circ q_B}&C\ar[d]^{\tau}\\
    \M(B\otimes\K)\ar[r]^{q_B}&\Q(B\otimes\K)
    }\]
    Furthermore, the algebra $E$ can be chosen so that $E$ has an approximate unit $\{p_n\}_{n=1}^\infty$ consisting of increasing sequence of projections $p_n<p_{n+1}$. 
\end{lem}
\begin{proof}
Applying the isomorphism $\eta_{-, -}$ in Sec. \ref{nati} and Lemma \ref{ade},
there exists an isomorphism
\[KK(SC, B)\ni \varphi\hat{\otimes}KK(i_{1_B\otimes e_0})^{-1}\mapsto [\tau]\in \operatorname{Ext}^1(C, B\otimes\K),\]
and we obtain an essential extension \[B\otimes\K\triangleleft E:=q_B^{-1}(\tau (C))\subset \M(B\otimes\K)\] with the following exact triangle
\[SC\xrightarrow{\varphi}B\otimes\K\to E\xrightarrow{(\tau^{-1}\circ q_{B})}C\]
where $\tau : C\hookrightarrow\Q(B\otimes\K)$ is the Busby invariant of the extension $E$.

We will modify $E$ so that $E$ admits the desired approximate unit.
Recall $[1_C]_0=0\in K_0(C)$ (i.e., $[\tau(1_C)]_0=0\in K_0(\Q(B\otimes\K))$).
By the standard argument on the Murray--von Neumann equivalence and the definition of $K_0$-group,
there is a unitary $U\in \mathbb{M}_3(\M(B\otimes\K))$ satisfying
\[(\id_{\mathbb{M}_3}\otimes q_B)(U)\left(\begin{array}{ccc}\tau(1_C)&0&0\\
0&1_{\Q(B\otimes\K)}&0\\
0&0&0
\end{array}\right)(\id_{\mathbb{M}_3}\otimes q_B)(U^*)=\left(\begin{array}{ccc}1_{\Q(B\otimes\K)}&0&0\\
0&1_{\Q(B\otimes\K)}&0\\
0&0&0
\end{array}\right).\]
For a unital $*$-homomorphism $\pi : C\to \M(B\otimes\K)$ with $\pi(C)\cap B\otimes\K=0$,
consider the following non-unital essential extension
\[\tilde{E}:=\left\{\left.U\left(\begin{array}{ccc}
   x  & B\otimes\K&B\otimes\K  \\
    B\otimes\K & y& B\otimes\K\\
    B\otimes\K& B\otimes\K&B\otimes\K
\end{array}\right)U^*\in\mathbb{M}_3(\M(B\otimes\K))\;\right| \; \begin{array}{cc}
     \tau^{-1}\circ q_B(x)=(q_B\circ \pi)^{-1}\circ q_B(y), \\\;\; x\in E,\;\;
     y\in \pi (C)+B\otimes\K
\end{array}   \right\},\]
\[B\otimes\mathbb{M}_3(\K)\to \tilde{E}\ni U\left(\begin{array}{ccc}
   x  & B\otimes\K&B\otimes\K  \\
    B\otimes\K & y& B\otimes\K\\
    B\otimes\K& B\otimes\K&B\otimes\K
\end{array}\right)U^*\mapsto \tau^{-1}(q_B(x))\in  C.\]
The Busby invariant of $\tilde{E}$ is given by $\ad q_B(U)\circ(\tau\oplus (q_B\circ\pi)\oplus 0)$.
Since $[\tau]=[\tau\oplus (q_B\circ\pi)]=[\ad (\id_{\mathbb{M}_3}\otimes q_B)(U)\circ(\tau\oplus (q_B\circ\pi)\oplus 0)]\in \operatorname{Ext}^1(C, B\otimes\K)$,
the proof of Lemma \ref{ade} again implies that
\[SC\xrightarrow{\varphi}B\otimes\mathbb{M}_3(\K)\to \tilde{E}\to C\]
is an exact triangle (One can also easily check that this new triangle is equivalent to the previous one).
Then, there are projections
\[p_n :=\left(\begin{array}{ccc}1_{\M(B\otimes\K)}&0&0\\
0&1_{\M(B\otimes\K)}&0\\
0&0&(1_B\otimes r_n)
\end{array}\right)\in \tilde{E}\]
where $r_n\in\K$ is the rank $n$ projection satisfying $r_n<r_{n+1}, SOT-\lim_{n\to \infty}r_n=1_{\M(\K)}$.
Fix an arbitrary $Y\in \tilde{E}$.
Since $\ad (\id_{\mathbb{M}_3}\otimes q_B)(U)\circ (\tau\oplus (q_B\circ\pi)\oplus 0)(1_C)=1_{\Q(B\otimes\K)}\oplus 1_{\Q(B\otimes\K)}\oplus 0$, one has \[\id_{\mathbb{M}_3}\otimes q_B(Y(0\oplus 0\oplus 1_{\M(B\otimes\K)}))=\id_{\mathbb{M}_3}\otimes q_B(Y)(1\oplus 1\oplus 0)(0\oplus 0\oplus 1)=0\] and $Y(0\oplus 0\oplus 1_{\M(B\otimes\K)})\in \mathbb{M}_3(B\otimes\K)$,
and the direct computation yields
\begin{align*}
    ||Y-Yp_n||=&||Y(1\oplus 1\oplus 1)-Y(1\oplus 1\oplus (1_B\otimes r_n))||\\
    =&||Y(0\oplus 0\oplus 1)-Y(0\oplus 0\oplus 1) (1\oplus 1\oplus (1_B\otimes r_n))||\\
    \to&0\quad (n\to\infty).
\end{align*}
Thus, the extension $\tilde{E}$ has the desired approximate unit.
\end{proof}

Now we have the following commutative diagram of exact triangles
\[\xymatrix{
SB\otimes\mathbb{K}\ar@<+0.8 ex>[r]^{-KK(Sl_E)}&SE\ar@<+0.8 ex>[r]^{-KK(S(\tau^{-1}\circ q_B))}&SC\ar@<+0.8 ex>[r]^{\varphi}&B\otimes\mathbb{K}\\
SB\ar[u]^{KK(Si_{(1_B\otimes e_0)})}\ar[r]^{-KK(S\iota_0)}&SA\ar[r]&\operatorname{Cone}(\iota_0)\ar[r]^{KK(ev_1)}\ar[u]^{KK(Sf)\hat{\otimes}(b\otimes I_{Cone})}&B,\ar[u]^{KK(i_{(1_B\otimes e_0)})}
}\]
and, by \cite[Sec. 2]{MN} and the suspention isomorphisms of KK-groups, we have a KK-equivalence $\eta\in KK(A, E)^{-1}$ making the following diagram commute 
\[
\xymatrix{
B\otimes \mathbb{K}\ar[r]&E\\
B\ar[u]^{KK(i_{(1_B\otimes e_0)})}\ar[r]^{KK(\iota_0)}&A.\ar[u]^{\eta}
}\]

\begin{proof}[{Proof of Theorem \ref{Mats}}]
    For an embedding $\iota_0 : B\to A$,
    Lemma \ref{extri} gives an essential, non-unital extension $B\otimes\K\triangleleft E$ satisfying the assumption 1., 2., 3., 4. in the previous section for $U_g\equiv 1$.
    Corollary \ref{yhe} and the above argument give the following commutative diagram
    \[\xymatrix{\mathbb{C}\quad\;\ar@<-1.5 ex>@{=}[d]\ar[r]&B\ar[d]_{KK(i_{(1_B\otimes e_0)})}\ar[r]^{KK(\iota_0)}&A\ar[d]^{\eta}\\
    \ar@<-1.5 ex>@{=}[d]&B\otimes\K\ar[r]&E\\
    \mathbb{C}\quad\;\ar[r]&B\ar[u]^{KK(i_{(1_B\otimes e_0)})}\ar[r]^{KK(\iota)}&A_E\ar[u]^{k_E}}\]
    Thus, the Kirchberg--Phillips theorem gives an isomorphism $\psi : A_E\to A$ with $KK(\psi)=k_E\hat{\otimes}\eta^{-1}$.
    Now we write $\iota_1:=\psi\circ\iota$, $\gamma_g:=\psi\circ\gamma^E_g\circ\psi^{-1}$ for which $KK(\iota_1)=KK(\iota_0)$ holds.
    Theorem \ref{km} shows $A^\gamma=\iota_1(B)$ and the existence of $G$-equivariant conditional expectation $A\to A^\gamma$.
    Since $U_g\equiv 1$, Corollary \ref{yhe} also shows $(A, \gamma)\cong (A_E, \gamma^E)\sim_{KK^G} (E, \id)\sim_{KK^G}(A, \id)$.
    Lemma \ref{out} shows that $\gamma_g, g\not=e$ is outer.
\end{proof}

\begin{remark}\label{ish}
For any countable infinite discrete groups $G_1, G_2$,
one has $l^2(G_1)\cong l^2(G_2)$ and the inclusions 
\[E\subset \mathcal{T}_{H\otimes l^2(G_1)\otimes E}, \quad E\subset\mathcal{T}_{H\otimes l^2(G_2)\otimes E}\]
are the same when we forget about the group actions.
    Thus, our construction actually shows the following.
    
    For a given unital embedding $\iota_0 : B\to A$,
    there exists another unital embedding $\iota_1 : B\to A$ with a conditional expectation $A\to \iota_1(B)$ satisfying
    \begin{enumerate}
        \item $KK(\iota_0)=KK(\iota_1)$,
        \item For any countable infinite discrete group $G$, there exists an action $\gamma$ such that $A^{\gamma}=\iota_1(B)$ and the above conditional expectation is $G$-invariant.
    \end{enumerate}
\end{remark}

\section{Proof of Theorem \ref{fee}}\label{busby}
In this section, the group $G$ is countable infinite discrete amenable, and $(A, \alpha)$ is a unital Kirchberg $G$-algebra.
We write $C_A:=\operatorname{Cone}(1_A)$ for short and denote by $C_\alpha$ the induced $G$-action on $C_A$ (i.e., $({C_\alpha})_g(a(t)):=\alpha_g(a(t))$, $a(t)\in C_A\subset C_0(0, 1]\otimes A$).
\subsection{Finding equivariant extension}
\begin{lem}[{\cite[Lem. 2.2.]{fkk}}]\label{Dad}
    There exists a unital, separable, nuclear $G$-algebra $(C, \sigma)$ with $G$-equivariant embedding 
    \[f : (C_A, C_\alpha)\hookrightarrow (C, \sigma)\]
    providing $KK^G$-equivalence $KK^G(f)\in KK^G((C_A, C_\alpha), (C, \sigma))^{-1}$.
\end{lem}
\begin{proof}
Let $\{S_n\}_{n=1}^\infty$ be the generating isometries of $\mathcal{O}_\infty$ with mutually orthogonal ranges.
    Let $p_0:=1-S_2S_2^*, p_1:=S_1S_1^*<p_0$ be the projections in $\mathcal{O}_\infty$.
Consider a $G$-equivariant embedding
\[f_1 : (C_A, C_\alpha)\ni x\mapsto x\otimes p_1\in (C_A\otimes (p_0\mathcal{O}_\infty p_0), C_\alpha\otimes\id).\]
    By the following diagram
    \[\xymatrix{
    (C_A, C_\alpha)\ar[dr]^{\id\otimes 1}\ar[r]^{-\otimes p_1\quad\quad\quad}&(C_A\otimes (p_1\mathcal{O}_\infty p_1), C_\alpha\otimes\id)\ar[r]^{Lem. \ref{mori}}&(C_A\otimes(p_0\mathcal{O}_\infty p_0), C_\alpha\otimes\id)\\
    &(C_A\otimes\mathcal{O}_\infty, C_\alpha\otimes\id),\ar[u]^{\ad (1\otimes S_1)}&
    }\]
    one can check that the $G$-equivariant embedding $(C_A, C_\alpha)\xrightarrow{f_1}(C_A\otimes(p_0\mathcal{O}_\infty p_0), C_\alpha\otimes\id)$ gives a $KK^G$-equivalence.
Since $p_0\mathcal{O}_\infty p_0$ is the Cuntz standard form,
there exists a unital embedding $\mathcal{O}_2\hookrightarrow p_0\mathcal{O}_\infty p_0$.
    By tensoring $(p_0\mathcal{O}_\infty p_0, \id)$ with the split exact sequence $C_A\to C_A^\sim\xrightarrow{\pi_A} \C$ of unitization, we have the following commutative diagram of the split exact sequences
    \[\xymatrix{
    (C_A\otimes(p_0\mathcal{O}_\infty p_0), C_\alpha\otimes\id)\ar[r]&((C_A^\sim)\otimes (p_0\mathcal{O}_\infty p_0), C_\alpha^\sim\otimes\id)\ar[r]^{\quad\quad\quad\pi}&  ((p_0\mathcal{O}_\infty p_0), \id)\\
    (C_A\otimes(p_0\mathcal{O}_\infty p_0), C_\alpha\otimes\id)\ar[r]\ar@{=}[u]&(\pi^{-1}(\mathcal{O}_2), C_\alpha^\sim\otimes\id)\ar[r]\ar[u]&(\mathcal{O}_2, \id)\ar@{^{(}->}[u]
    }\]
where $C_\alpha^\sim$ is the canonically induced action on $C_A^\sim$ and $\pi:=\pi_A\otimes\id_{p_0\mathcal{O}_\infty p_0} : C_A^\sim\otimes (p_0\mathcal{O}_\infty p_0)\to p_0\mathcal{O}_\infty p_0$.
Define $(C, \sigma):=(\pi^{-1}(\mathcal{O}_2), C_\alpha^\sim\otimes\id)$ so that $(C, \sigma)$ is a unital separable nuclear $G$-algebra.
    Since $(\mathcal{O}_2, \id)\sim_{KK^G}0$,
    the equivariant inclusion 
    \[f : (C_A, C_\alpha)\xrightarrow{f_1}(C_A\otimes(p_0\mathcal{O}_\infty p_0),  C_\alpha\otimes\id)\to (C, \sigma)\]
    is a $KK^G$-equivalence.
\end{proof}

Let $P_\infty$ be the unital Kirchberg algebra satisfying $P_\infty\sim_{KK} S$.
Let $\K\to\mathcal{T}_0\to P_\infty$ be a unital essential extension satisfying $\mathcal{T}_0\sim_{KK}0$.
The existence of $\mathcal{T}_0$ follows from Lemma \ref{ade}.
One has $(\mathcal{T}_0, \id)\sim_{KK^G} 0$.

In the previous sections,
we have used extension $\K\to E\to {^{"}SC_A^"}$ to construct an ergodic action of $A$.
In this section, we will find an extension $\mathcal{O}_\infty\otimes\K\to \mathcal{E}\to P_\infty\otimes C\sim_{KK^G}SC_A$ by using $C$ in the above lemma.

We write $(\mathcal{O}_\infty^s, \id\otimes\lambda):=(\K\otimes\mathcal{O}_\infty, \id_\K\otimes\lambda)$ for short.
By tensoring $(\K_G\otimes\mathcal{O}_\infty^s, \rho\otimes(\id\otimes\lambda))$ with the extension $\K\to\mathcal{T}_0\to P_\infty$,
we consider the exact triangle
\[\xymatrix{
S(P_\infty\otimes \K_G\otimes\mathcal{O}_\infty^s, \id\otimes \rho\otimes\id\otimes\lambda)\ar@{-->}[r]^{\beta_1}&(\K\otimes \K_G\otimes \mathcal{O}_\infty^s, \id\otimes \rho\otimes (\id \otimes\lambda))-\\
\to(\mathcal{T}_0\otimes\K_G\otimes\mathcal{O}_\infty^s, \id\otimes\rho\otimes (\id\otimes\lambda))\ar@<+0.8 ex>[r]^{\pi_{\mathcal{T}_0\otimes\K_G\otimes\mathcal{O}_\infty^s}}&(P_\infty\otimes \K_G\otimes \mathcal{O}_\infty^s, \id\otimes\rho\otimes(\id\otimes\lambda)), 
}\]
where \[\beta_1:=KK^G(\iota(\mathcal{T}_0\otimes\K_G\otimes\mathcal{O}_\infty^s))\hat{\otimes}KK^G( j(\mathcal{T}_0\otimes\K_G\otimes\mathcal{O}_\infty^s))^{-1}\] is well-defined by Theorem \ref{wkeq} and makes the following diagram commute
\[\xymatrix{SP_\infty\otimes\K_G\otimes\mathcal{O}_\infty^s\ar@{-->}[r]^{\beta_1}\ar[dr]_{\iota({\mathcal{T}_0\otimes\K_G\otimes\mathcal{O}_\infty^s})\quad\quad\quad}&\K\otimes\K_G\otimes\mathcal{O}_\infty^s\ar[d]^{j({\mathcal{T}_0\otimes\K_G\otimes\mathcal{O}_\infty^s})}\\
&\operatorname{Cone}(\pi_{\mathcal{T}_0\otimes\K_G\otimes\mathcal{O}_\infty^s}).
}\]

Since $(\mathcal{T}_0\otimes\K_G\otimes\mathcal{O}_\infty^s, \id\otimes\rho\otimes\id\otimes\lambda)\sim_{KK^G}(\mathcal{T}_0, \id)\sim_{KK^G}0$,
 $\beta_1$ is a $KK^G$-equivalence.
where $j({\mathcal{T}_0\otimes\K_G\otimes\mathcal{O}_\infty^s})$ is a $KK^G$-equivalence by Theorem \ref{wkeq}.
Let \[\beta_2\in KK(\C, SP_\infty)^{-1}\subset KK^G((\C, \id), (SP_\infty, \id))^{-1}\] be a non-equivariant Bott element.
Recall the $KK^G$-equivalence $\mu_\C\in KK^G((\C, \id), (\K_G, \rho))^{-1}$ appearing right before Lemma \ref{mori}. 
By the suspension isomorphism $KK^G((C, \sigma), (\mathcal{O}_\infty^s, \id\otimes\lambda))\cong KK^G((SP_\infty\otimes C, \id\otimes\sigma), (SP_\infty\otimes\mathcal{O}_\infty^s, \id\otimes(\id\otimes\lambda)))$ and the diagram
\[\xymatrix{
&C_A\ar[r]^{KK^G(ev_1)}\ar[d]^{KK^G(f)}&\C\ar[d]^{KK^G(1_{\mathcal{O}_\infty})}\\
&C\ar[d]^{\beta_2\otimes I_C}&\mathcal{O}_\infty\ar[d]^{KK^G(i_{e\otimes e}\otimes \id_{\mathcal{O}_\infty})}\\
&SP_\infty\otimes C\ar[dl]^{\mu_\C}\ar@{-->}[d]^{I_{SP_\infty}\otimes KK^G(\varphi, u)}&\K\otimes\mathcal{O}_\infty^s\ar[d]^{\mu_\C}\\
SP_\infty\otimes\K_G\otimes C\ar@{-->}[d]_{I_{SP_\infty}\otimes KK^G(\ad W^*\circ(\id\otimes\varphi))}&SP_\infty\otimes\mathcal{O}_\infty^s\ar[dl]^{\mu_\C}&\K\otimes\K_G\otimes\mathcal{O}_\infty^s\ar@{=}[d]\\
SP_\infty\otimes\K_G\otimes\mathcal{O}_\infty^s\ar@{-}[r]&\ar[r]^{\beta_1\quad\quad\quad\quad\quad\quad\quad\quad\quad}&\K\otimes\K_G\otimes\mathcal{O}_\infty^s,
\quad \cdots(*)}\]
Theorem \ref{yeah} gives a cocycle embedding
\[(\varphi, u_\cdot) : C\hookrightarrow \mathcal{O}_\infty^s,\quad u_g\in\M(\mathcal{O}_\infty^s)\]
making the right large circuit commutes,
and Lemma \ref{IM} gives a unitary $W\in \M(\K_G\otimes\mathcal{O}_\infty^s)$ and $G$-equivariant embedding
\[\Phi:=\ad W^*\circ (\id\otimes\varphi) : \K_G\otimes C\hookrightarrow\K_G\otimes\mathcal{O}_\infty^s\]
making the left square commute.
We write
\[k_1:=KK^G(f)\hat{\otimes}(\beta_2\otimes I_C)\hat{\otimes}\mu_{\C}\in KK^G((C_A, C_\alpha), (SP_\infty\otimes\K_G\otimes C, \id_{SP_\infty}\otimes\rho\otimes\sigma))^{-1},\]
\[k_2:=KK^G(1_{\mathcal{O}_\infty})\hat{\otimes}KK^G(i_{e\otimes e})\hat{\otimes} \mu_{\C}\in KK^G((\C, \id), (\K\otimes\K_G\otimes\mathcal{O}_\infty^s, \id\otimes\rho\otimes(\id\otimes\lambda)))^{-1}.\]
We obtain a "Busby invariant" 
\[\id_{P_\infty}\otimes\Phi : P_\infty\otimes \K_G\otimes C \hookrightarrow P_\infty\otimes \K_G\otimes\mathcal{O}_\infty^s\]
and an induced essential extension
\[\mathcal{E}:=\pi_{\mathcal{T}_0\otimes\K_G\otimes\mathcal{O}_\infty^s}^{-1}(P_\infty\otimes \Phi(\K_G\otimes C))=\K\otimes(\K_G\otimes\mathcal{O}_\infty^s)+\mathcal{T}_0\otimes\Phi(\K_G\otimes C),\]
\[\K\otimes\K_G\otimes\mathcal{O}_\infty^s\xrightarrow{l_\mathcal{E}} \mathcal{E}\xrightarrow{\pi_\mathcal{E}}P_\infty\otimes \K_G\otimes\mathcal{O}_\infty^s.\]
\begin{lem}\label{appup}
    The algebra $\mathcal{E}$ has an approximate unit $\{p_n\}_{n=1}^\infty$ consisting of increasing sequence of projections $p_n<p_{n+1}$.
\end{lem}
\begin{proof}
    By the definition of $\Phi$,
    one has \[\mathcal{E}\cong \ad(1_{\M(\K)}\otimes W)(\mathcal{E})=\K\otimes\K_G\otimes\mathcal{O}_\infty^s+\mathcal{T}_0\otimes\K_G\otimes \varphi(C)\cong\K_G\otimes(\K\otimes\mathcal{O}_\infty^s+\mathcal{T}_0\otimes\varphi(C)).\]
    Thus, it is enough to show that $\K\otimes\mathcal{O}_\infty^s+\mathcal{T}_0\otimes\varphi(C)$ has an approximate unit consisting of projections.
    We write $r:=1_{\mathcal{T}_0}\otimes\varphi(1_C)\in\M(\K\otimes\mathcal{O}_\infty^s)$.
    Since $1_{\M(\mathcal{O}_\infty^s)}-\varphi(1_C)\not=0$, one has $1_{\M(\K\otimes\mathcal{O}_\infty^s)}-r\not=0$.
    So $(1-r)\K\otimes\mathcal{O}_\infty^s(1-r)\cong \K\otimes\mathcal{O}_\infty$ holds and there is an approximate unit $\{r_n\}_{n=1}^\infty\subset (1-r)\K\otimes\mathcal{O}_\infty^s(1-r)$ consisting of increasing sequence of projections $r_n<r_{n+1}$.
    The elements $p_n:=r+r_n\in\mathcal{E}$ give the desired approximate unit.
\end{proof}
Since $\Phi$ is $G$-equivariant,
$(\mathcal{E}, \id_{\mathcal{T}_0}\otimes\rho\otimes(\id\otimes\lambda)|)$ is a well-defined separable nuclear $G$-subalgebra of $(\mathcal{T}_0\otimes(\K_G\otimes\mathcal{O}_\infty^s), \id_{\mathcal{T}_0}\otimes (\rho\otimes (\id\otimes\lambda)))$.
By Theorem \ref{wkeq},
one has an exact triangle
\[SP_\infty\otimes\K_G\otimes C\xrightarrow{\beta_3}\K\otimes\K_G\otimes\mathcal{O}_\infty^s\xrightarrow{l_\mathcal{E}}\mathcal{E}\xrightarrow{\pi_\mathcal{E}}P_\infty\otimes\K_G\otimes C,\]
where $\beta_3$ is defined by Theorem \ref{wkeq} and the following diagram
\[\xymatrix{
SP_\infty\otimes\K_G\otimes C\ar@{-->}[r]^{\beta_3}\ar[dr]^{\iota({\mathcal{E}})}&\K\otimes\K_G\otimes\mathcal{O}_\infty^s\ar[d]^{j({\mathcal{E})}}\\
&\operatorname{Cone}(\pi_\mathcal{E}).
}\]
\begin{lem}
    The following diagram commutes in $KK^G$
    \[\xymatrix{
    C_A\ar[r]^{KK^G(ev_1)}\ar[d]^{k_1}&\C\ar[d]^{k_2}\\
    SP_\infty\otimes\K_G\otimes C\ar[r]^{\beta_3}&\K\otimes\K_G\otimes\mathcal{O}_\infty^s.
    }\]
\end{lem}
\begin{proof}
    There is a natural map
    \[\theta : \operatorname{Cone}(\pi_\mathcal{E})\ni (c(t), d)\mapsto (\id_{P_\infty}\otimes\Phi(c(t)), d)\in\operatorname{Cone}(\pi_{\mathcal{T}_0\otimes\K_G\otimes\mathcal{O}_\infty^s})\]
    satisfying
    \[\theta\circ j(\mathcal{E})(x)=(0, x)=j({\mathcal{T}_0\otimes\K_G\otimes\mathcal{O}_\infty^s})(x),\quad x\in\K\otimes\K_G\otimes\mathcal{O}_\infty^s.\]
    Thus,
    one has the following commutative diagram
\[\xymatrix{
&\K\otimes\K_G\otimes\mathcal{O}_\infty^s\ar[d]^{KK^G(j(\mathcal{E}))}\\
SP_\infty\otimes\K_G\otimes C\ar[d]_{S\id_{P_\infty}\otimes\Phi}\ar[ur]^{\beta_3}\ar[r]&\operatorname{Cone}(\pi_\mathcal{E})\ar[d]^{KK^G(\theta)}\\
SP_\infty\otimes\K_G\otimes\mathcal{O}_\infty^s\ar[r]\ar[dr]^{\beta_1}&\operatorname{Cone}(\pi_{\mathcal{T}_0\otimes\K_G\otimes\mathcal{O}_\infty^s})\\
&\K\otimes\K_G\otimes\mathcal{O}_\infty^s.\ar[u]_{KK^G(j({\mathcal{T}_0\otimes\K_G\otimes\mathcal{O}_\infty^s}))}
}\]
By the diagram $(*)$ and definition of $k_1, k_2$, we have
\[k_1\hat{\otimes}\beta_3=k_1\hat{\otimes}(I_{SP_\infty}\otimes KK^G(\Phi))\hat{\otimes}\beta_1=KK^G(ev_1)\hat{\otimes}k_2.\]
\end{proof}
Applying Lemma \ref{benri} for the exact triangles
\[\xymatrix{
S\C\ar[r]^{-KK^G(S1_A)}\ar[d]^{Sk_2}&SA\ar[r]&C_A\ar[r]^{KK^G(ev_1)}\ar[d]^{k_1}&\C\ar[d]^{k_2}\\
S\K\otimes\K_G\otimes\mathcal{O}_\infty^s\ar@<+0.8 ex>[r]^{-KK^G(Sl_\mathcal{E})}&S\mathcal{E}\ar@<+0.8 ex>[r]^{-KK^G(S\pi_\mathcal{E})}&SP_\infty\otimes\K_G\otimes C\ar[r]^{\beta_3}&\K\otimes\K_G\otimes\mathcal{O}_\infty^s,
}\]
we have the following corollary.
\begin{cor}
There exists a $KK^G$-equivalence \[k_3\in KK^G((A, \alpha), (\mathcal{E}, \id_{\mathcal{T}_0}\otimes\rho\otimes (\id\otimes\lambda)))^{-1}\]
making the following diagram commute
\[\xymatrix{
(\C, \id)\ar[r]^{KK^G(1_A)}\ar[d]^{k_2}&(A, \alpha)\ar[d]^{k_3}\\
(\K\otimes\K_G\otimes\mathcal{O}_\infty^s, \id\otimes\rho\otimes(\id\otimes\lambda))\ar[r]&(\mathcal{E}, \id\otimes\rho\otimes(\id\otimes\lambda)).
}\]
\end{cor}
We write
\[\K\otimes\K_G\otimes\mathcal{O}_\infty^s=\K\otimes\K_G\otimes(\K\otimes\mathcal{O}_\infty)=\K(H_1\otimes (l^2(\mathbb{N})\otimes l^2(G))\otimes H_2)\otimes\mathcal{O}_\infty\]
where the $G$-action $(\id\otimes\rho\otimes(\id\otimes\lambda))_g$ is identified with $\left(\ad(1_{H_1}\otimes (1\otimes\rho_g)\otimes 1_{H_2})\right)\otimes\lambda_g$.
For the Hilbert space
\[\mathcal{H}:=(H_1\otimes (l^2(\mathbb{N})\otimes l^2(G))\otimes H_2)\oplus \C,\]
and the unitary
\[\mathcal{U}_g:=(1_{H_1}\otimes (1\otimes\rho_g)\otimes 1_{H_2})\oplus 1_\C\in \mathbb{B}(\mathcal{H}),\]
consider the full corner embeddings
\[i_1 : (\K\otimes\K_G\otimes\mathcal{O}_\infty^s, \id\otimes\rho\otimes(\id\otimes\lambda))\hookrightarrow(\K(\mathcal{H})\otimes\mathcal{O}_\infty, \ad\mathcal{U}\otimes\lambda),\]
\[i_2 : (\mathcal{E}, \id\otimes\rho\otimes(\id\otimes\lambda))\hookrightarrow(\mathcal{E}+(\K(\mathcal{H})\otimes\mathcal{O}_\infty), \ad\mathcal{U}\otimes\lambda).\]
By Lemma \ref{mori},
$KK^G(i_1), KK^G(i_2)$ are $KK^G$-equivalences and 
\[\K(\mathcal{H})\otimes\mathcal{O}_\infty\to (\mathcal{E}+(\K(\mathcal{H})\otimes\mathcal{O}_\infty))\to P_\infty\otimes \Phi(\K_G\otimes C)\]
is an essential extension with a $G$-invariant projection
$e_0\otimes 1_{\mathcal{O}_\infty}$ where 
\[e_0 : \mathcal{H}=(H_1\otimes(l^2(\mathbb{N})\otimes l^2(G))\otimes H_2)\oplus\C\to 0\oplus \C\] is a minimal projection of $\K(\mathcal{H})$ satisfying $\mathcal{U}_g e_0=e_0=e_0\mathcal{U}_g^*$.

Now we obtain the following.
\begin{cor}\label{fext}
    There exists an essential extension
    \[\K(\mathcal{H})\otimes\mathcal{O}_\infty\triangleleft\mathcal{E}+(\K(\mathcal{H})\otimes\mathcal{O}_\infty)=:\mathcal{E}'\]
    satisfying
    \begin{enumerate}
        \item $(\mathcal{E}', (\ad \mathcal{U})\otimes\lambda)$ is a non-unital separable nuclear $G$-algebra with an approximate unit consisting of an increasing sequence of projections,
        \item $((\ad \mathcal{U}_g)\otimes\lambda_g)(e_0\otimes 1_{\mathcal{O}_\infty})=e_0\otimes 1_{\mathcal{O}_\infty}\in \mathcal{E}'$,
        \item There exist $KK^G$-equivalences
        \[k_4:=k_2\hat{\otimes}KK^G(i_1)\in KK^G((\C, \id), (\K(\mathcal{H})\otimes\mathcal{O}_\infty, (\ad\mathcal{U})\otimes\lambda))^{-1},\]
        \[k_5:=k_3\hat{\otimes} KK^G(i_2)\in KK^G((A, \alpha), (\mathcal{E}', (\ad\mathcal{U})\otimes\lambda))^{-1}\]
        making the following diagram commute
        \[\xymatrix{
        (\C, \id)\ar[r]^{KK^G(1_A)}\ar[d]^{k_4}&(A, \alpha)\ar[d]^{k_5}\\
        (\K(\mathcal{H})\otimes\mathcal{O}_\infty, (\ad\mathcal{U})\otimes\lambda)\ar[r]&(\mathcal{E}', (\ad\mathcal{U})\otimes\lambda).
        }\]
        \item $F(k_4)=KK(\C(e_0\otimes 1_{\mathcal{O}_\infty})\hookrightarrow \K(\mathcal{H})\otimes\mathcal{O}_\infty)=KK(i_{e_0\otimes 1_{\mathcal{O}_\infty}})$.
        \item $\mathcal{E}'$ has a non-degenerate faithful $G$-equivariant representation
        \[(\mathcal{E}', (\ad\mathcal{U})\otimes\lambda)\subset (\mathbb{B}(\mathcal{H}\otimes\mathcal{F}(l^2(G))), \ad (\mathcal{U}\otimes\mathcal{F}(\lambda))),\]
        where $\mathcal{F}(\lambda_g)$ is the implementing unitary on the Fock space $\mathcal{F}(l^2(G))=\C\Omega_{l^2(G)}\oplus\bigoplus_{k=1}^\infty l^2(G)^{\otimes k}$ of the quasi-free automorphism $\lambda_g\in\operatorname{Aut}(\mathcal{O}_\infty)$ (i.e., $(\mathcal{O}_\infty, \lambda)\subset (\mathbb{B}(\mathcal{F}(l^2(G))), \ad\mathcal{F}(\lambda))$).
    \end{enumerate}
\end{cor}
\subsection{Pimsner construction for Theorem \ref{fee}}
We apply A. Kumujian's construction for $\mathcal{E}'\subset\mathbb{B}(\mathcal{H}\otimes\mathcal{F}(l^2(G)))$ in Corollary \ref{fext}.
Consider Hilbert $\mathcal{E}'-\mathcal{E}'$-bimodule
\[\mathcal{X}:=(\mathcal{H}\otimes \mathcal{F}(l^2(G)))\otimes_\C l^2(G)\otimes_\C \mathcal{E}'\]
with a $G$-action
\[g\cdot (\zeta\otimes\delta_h\otimes x):=(\mathcal{U}_g\otimes\mathcal{F}(\lambda_g))(\zeta)\otimes\delta_{gh}\otimes g\cdot x,\quad g\cdot x:=\ad (\mathcal{U}_g\otimes \mathcal{F}(\lambda_g))(x),\quad \zeta\in \mathcal{H}\otimes\mathcal{F}(l^2(G)),\; x\in\mathcal{E}'.\]
We obtain Toeplitz--Pimsner algebra $\mathcal{T}_{\mathcal{X}}$ with the quasi-free action
\[\Gamma_g(T_{\zeta\otimes\delta_h\otimes x}):=T_{g\cdot(\zeta\otimes\delta_h\otimes x)}.\]
By Lemma \ref{appup} and the proof of Lemma \ref{kir},
$\mathcal{T}_{\mathcal{X}}$ is a stable Kirchberg algebra with the following commutative diagram in $KK^G$
\[\xymatrix{
(\K(\mathcal{H})\otimes\mathcal{O}_\infty, (\ad\mathcal{U})\otimes\lambda)\ar[r]&(\mathcal{E}', (\ad\mathcal{U})\otimes\lambda)\ar[d]\\
(\C(e_0\otimes 1_{\mathcal{O}_\infty}), \id)\ar[u]^{KK^G(i_{e_0\otimes 1_{\mathcal{O}_\infty}})}\ar@{=}[d]\ar[r]&(\mathcal{T}_{\mathcal{X}}, \Gamma)\\
(\C(e_0\otimes 1), \id)\ar[r]&((e_0\otimes 1)(\mathcal{T}_{\mathcal{X}})(e_0\otimes 1), \Gamma |)\ar[u]^{KK^G(i_{e_0\otimes 1_{\mathcal{O}_\infty}})}
}\]
where $(\mathcal{E}', (\ad\mathcal{U})\otimes\lambda)\to(\mathcal{T}_\mathcal{X}, \Gamma)$ is the morphism induced by the natural embedding $\mathcal{E}'\hookrightarrow \mathcal{T}_\mathcal{X}$ (see Theorem \ref{wk}) and every vertical arrow is $KK^G$-equivalence.
Combining the above diagram with Theorem \ref{KP} and Corollary \ref{fext},
we obtain an isomorphism $\theta : A\to (e_0\otimes 1)(\mathcal{T}_\mathcal{X})(e_0\otimes 1)$ satisfying 
\[KK(\theta)=F(k_5)\hat{\otimes}KK(\mathcal{E}'\hookrightarrow \mathcal{T}_\mathcal{X})\hat{\otimes}KK(i_{e_0\otimes 1_{\mathcal{O}_\infty}})^{-1}.\]
We define an action $\gamma : G\curvearrowright A$ by
$\gamma_g :=\theta^{-1}\circ\Gamma |_g\circ \theta\in\operatorname{Aut}(A), \; g\in G$.


Now we obtain $KK^G$-equivalences \[k_6:=k_4\hat{\otimes} KK^G(i_{e_0\otimes 1_{\mathcal{O}_\infty}})^{-1}\in KK^G((\C, \id), (\C, \id))^{-1},\]
\[ k_7:=k_5\hat{\otimes} KK^G(\mathcal{E}'\to \mathcal{T}_\mathcal{X})\hat{\otimes}KK^G(i_{e_0\otimes 1_{\mathcal{O}_\infty}})^{-1}\hat{\otimes} KK^G(\theta^{-1})\in KK^G((A, \alpha), (A, \gamma))^{-1}\]
with the commutative diagram
\[\xymatrix{
(\C, \id)\ar[r]^{KK^G(1_A)}\ar[d]^{k_6}&(A, \alpha)\ar[d]^{k_7}\\
(\C, \id)\ar[r]^{KK^G(1_A)}&(A, \gamma).
}\]
By Corollary \ref{fext}, $F(k_6)=F(k_4)\hat{\otimes} KK(i_{e_0\otimes 1_{\mathcal{O}_\infty}})^{-1}=KK(\id_\C)$,
and one has \[KK(1_A)\hat{\otimes}F(k_7)=KK(1_A).\]
By the Pimsner construction, one has the following $G$-invariant state
\[A=(e_0\otimes 1)(\mathcal{T}_\mathcal{X})(e_0\otimes 1)\xrightarrow{\langle -\Omega_\mathcal{X}, \Omega_{\mathcal{X}}\rangle}e_0\otimes\mathcal{O}_\infty\xrightarrow{\langle-\Omega_{l^2(G)}, \Omega_{l^2(G)}\rangle}\C.\]
\begin{proof}[{Proof of Theorem \ref{fee}}]
To complete the proof, it is enough to show that $(A, \gamma)$ is ergodic and point-wise outer.

First, we show that $\gamma_g, \;g\not=e$ is outer.
Fix $\zeta_0\in \mathcal{H}, \; e_0(\zeta_0)=\zeta_0, 
 \; ||\zeta_0||=1$.
Assume that there is $V\in U((e_0\otimes 1)(\mathcal{T}_{\mathcal{X}})(e_0\otimes 1))$ with $\gamma_g=\ad V$ for some $g\in G\backslash\{e\}$.
Note that $(e_0\otimes 1)(\mathcal{T}_{\mathcal{X}})(e_0\otimes 1)$ is faithfully represented on \[(e_0\otimes 1)\mathcal{E}'\Omega_\mathcal{X}\oplus\bigoplus_{k=1}^\infty((\C\zeta_0\otimes \mathcal{F}(l^2(G)))\otimes l^2(G)\otimes\mathcal{E}')\otimes_{\mathcal{E}'}\mathcal{X}^{\otimes_{\mathcal{E}'}k-1}.\]
The same argument as in the proof of Lemma \ref{out} shows that $\beta_z(V)=V, \;\;z\in\mathbb{T}$.
There exist a finite set $F\subset G$ and finitely many elements $\mu_i, \nu_i, \xi_i, \eta_i\in \bigcup_{k=1}^\infty((\mathcal{H}\otimes\mathcal{F}(l^2(G)))\otimes l^2(F)\otimes \mathcal{E}'^{\otimes k})$, $v\in \mathcal{E}'$ satisfying
\[||V-(e_0\otimes 1)(v+\sum_i(T_{\mu_i}+T_{\nu_i}^*+T_{\xi_i}T_{\eta_i}^*))(e_0\otimes 1)||<1/8.\]
Applying $\int_\mathbb{T}\beta_z(\cdot)dz$,
we may assume
\[||V-(e_0\otimes 1)(v+\sum_iT_{\xi_i}T_{\eta_i}^*)(e_0\otimes 1)||<1/8.\]
For $\zeta_0\otimes\Omega_{l^2(G)}\in \mathcal{H}\otimes\mathcal{F}(l^2(G))$,
we have $\mathcal{U}_g\otimes\mathcal{F}(\lambda_g)(\zeta_0\otimes\Omega_{l^2(G)})=\zeta_0\otimes\Omega_{l^2(G)}$.
Since $|G|=\infty$, there is $h\in G\backslash F$,
and the direct computation yields
\begin{align*}
    (\zeta_0\otimes\Omega_{l^2(G)})\otimes\delta_{gh}\otimes (e_0\otimes 1)=&(\mathcal{U}_g\otimes\mathcal{F}(\lambda_g))(\zeta_0\otimes\Omega_{l^2(G)})\otimes\delta_{gh}\otimes g\cdot (e_0\otimes 1)\\
    =&g\cdot((\zeta_0\otimes\Omega_{l^2(G)})\otimes\delta_h\otimes (e_0\otimes 1))\\
    =&\Gamma_g(T_{(\zeta_0\otimes\Omega_{l^2(G)})\otimes \delta_h\otimes (e_0\otimes 1)})(e_0\otimes 1)\Omega_\mathcal{X}\\
    =&\gamma_g((e_0\otimes 1)T_{(\zeta_0\otimes\Omega_{l^2(G)})\otimes \delta_h\otimes(e_0\otimes 1)}(e_0\otimes 1))(e_0\otimes 1)\Omega_\mathcal{X}\\
    =&VT_{(\zeta_0\otimes\Omega_{l^2(G)})\otimes\delta_h\otimes(e_0\otimes 1)}V^*(e_0\otimes 1)\Omega_\mathcal{X}\\
    \approx_{1/8}&VT_{(\zeta_0\otimes\Omega_{l^2(G)})\otimes\delta_h\otimes (e_0\otimes 1)}(e_0\otimes 1)v^*(e_0\otimes 1)\Omega_\mathcal{X}\\
    \approx_{1/8(1+1/8)}&((e_0\otimes 1)v(\zeta_0\otimes\Omega_{l^2(G)}))\otimes \delta_h\otimes ((e_0\otimes 1)v^*(e_0\otimes 1))\\
    &+\sum_i (e_0\otimes 1)T_{\xi_i}T_{\eta_i}^*((\zeta_0\otimes \Omega_{l^2(G)})\otimes \delta_h\otimes ((e_0\otimes 1)v^*(e_0\otimes 1))).
\end{align*}
Since $h\not\in F$,
one has $\sum_i T_{\xi_i}T_{\eta_i}^*((\zeta_0\otimes \Omega_{l^2(G)})\otimes \delta_h\otimes ((e_0\otimes 1)v^*(e_0\otimes 1)))=0$ which implies
\begin{align*}
    1=&||\langle(\zeta_0\otimes \Omega_{l^2(G)})\otimes\delta_{gh}\otimes(e_0\otimes 1), (\zeta_0\otimes\Omega_{l^2(G)})\otimes\delta_{gh}\otimes (e_0\otimes 1)\rangle||\\
    \approx_{1/2}&||\langle((e_0\otimes 1)v(\zeta_0\otimes\Omega_{l^2(G)}))\otimes \delta_h\otimes ((e_0\otimes 1)v^*(e_0\otimes 1)), (\zeta_0\otimes\Omega_{l^2(G)})\otimes\delta_{gh}\otimes (e_0\otimes 1)\rangle||=0.
\end{align*}
This is a contradiction and $\gamma_g$ must be outer.

Next, we show $A^\gamma=\C$.
Fix $Y\in ((e_0\otimes 1)(\mathcal{T}_{\mathcal{X}})(e_0\otimes 1))^\gamma$ and arbitrary $\epsilon>0$.
There exist a finite set $F_1'\subset G$ and finitely many elements $\mu_i, \nu_i, \xi_i, \eta_i\in \bigcup_{k=1}^\infty ((\mathcal{H}\otimes \mathcal{F}(l^2(G))\otimes l^2(F_1')\otimes\mathcal{E}')^{\otimes k}$, $Y'\in \mathcal{E}'$ satisfying
\[||Y-(e_0\otimes 1)(Y'+\sum_i(T_{\mu_i}+T_{\nu_i}^*+T_{\xi_i}T_{\eta_i}^*))(e_0\otimes 1)||<\epsilon.\]
Recall that the creation operators on $\mathcal{F}(l^2(G))$ give the generating isometries $\{S_g\}_{g\in G}$ of $\mathcal{O}_\infty$ with mutually orthogonal ranges.
Since $(e_0\otimes 1)Y'(e_0\otimes 1)\in \C e_0\otimes\mathcal{O}_\infty$,
there exist a finite set $F_1\supset F_1'$ and finitely many elements $\mu_j', \nu_j', \xi_j', \eta_j'\in\bigcup_{k=1}^\infty l^2(F_1)^{\otimes k}$, $y\in \C$
satisfying
\[||(e_0\otimes 1)Y'(e_0\otimes 1)-e_0\otimes (y+\sum_j(S_{\mu_j'}+S_{\nu_j'}^*+S_{\xi_j'}S_{\eta_j'}^*))||<\epsilon.\]

Consider a faithful representation (see \cite[p. 206]{P})
\[\pi : \mathcal{T}_{\mathcal{X}}\ni T_\zeta\mapsto T_\zeta |\in \mathcal{L}_{\mathcal{E}'}(\bigoplus_{k=1}^\infty\mathcal{X}^{\otimes_{\mathcal{E}'}k}).\]
Then, there exist a finite set $F_2\subset G$ and
\[u, v\in \bigoplus_{k=1}^\infty ((\mathcal{H}\otimes\mathcal{F}(l^2(F_2)))\otimes l^2(F_2)\otimes\mathcal{E}')^{\otimes_{\mathcal{E}'} k},\quad ||u||, ||v||\leq 1,\]
\[||Y-e_0\otimes y||\approx_\epsilon ||\langle \pi(Y-e_0\otimes y)u, v\rangle||.\]
Since $|G|=\infty$, there exists $g\in G$ with $g\cdot F_1\cap F_2=\emptyset$,
and one has
\[\pi(\gamma_g(e_0\otimes S_{\mu_j'}^*))v=\pi(\gamma_g(e_0\otimes S_{\nu_j'}^*))u=\pi(\gamma_g(e_0\otimes S_{\eta_j'}))u=0,\]
\[\pi(\Gamma_g(T_{\mu_i}^*)(e_0\otimes 1))v=\pi(\Gamma_g(T_{\nu_i}^*)(e_0\otimes 1))u=\pi (\Gamma_g(T_{\eta_i}^*)(e_0\otimes 1))u=0.\]
The direct computation yields
\begin{align*}
    &||Y-e_0\otimes y||\\
    \approx_\epsilon&||\langle \pi(Y-e_0\otimes y)u, v\rangle||\\
    =&||\langle \pi(\gamma_g(Y-e_0\otimes y))u, v\rangle||\\
    \approx_{2\epsilon}&||\langle \pi(\Gamma_g\left(e_0\otimes \left(\sum_j S_{\mu_j'}+S_{\nu_j'}^*+S_{\xi_j'}S_{\eta_j'}^*\right)+(e_0\otimes 1)\left(\sum_iT_{\mu_i}+T_{\nu_i}^*+T_{\xi_i}T_{\eta_i}^*\right)(e_0\otimes 1))\right))u, v\rangle||\\
    =&0.
\end{align*}
By the above argument, there exists $y_\epsilon\in\C$ with $||Y-e_0\otimes y_\epsilon||<3\epsilon$ for any $\epsilon>0$.
Thus, we can conclude $Y\in\C(e_0\otimes 1_{\mathcal{O}_\infty})=\C 1_A$ (i.e., $A^\gamma=\C1_A$).
\end{proof}

\end{document}